\documentclass[11pt,a4paper]{amsart}

\usepackage[margin=1in]{geometry}
\usepackage[T1]{fontenc}
\usepackage{lmodern}
\usepackage{amsmath,amssymb,mathtools,mathrsfs}
\usepackage{microtype}
\usepackage{xcolor}
\newif\ifrevisionblue
\revisionbluefalse
\DeclareRobustCommand{\revision}[1]{{\ifrevisionblue\color{blue}\fi#1}}
\usepackage[hidelinks]{hyperref}
\hypersetup{
  pdftitle={Vopěnka's Principle without Choice: Preservation under Symmetric Extensions},
  pdfauthor={Tom de Groot and Wojciech Aleksander Woloszyn},
  pdfsubject={Set theory},
  pdfkeywords={Vopenka principle, symmetric extensions, axiom of choice}
}

\newcommand{\VP}{\mathsf{VP}}
\newcommand{\AC}{\mathsf{AC}}
\newcommand{\ZF}{\mathsf{ZF}}
\newcommand{\ZFC}{\mathsf{ZFC}}
\newcommand{\GBC}{\mathsf{GBC}}
\newcommand{\HS}{\mathsf{HS}}
\newcommand{\Ord}{\mathrm{Ord}}
\newcommand{\Def}{\operatorname{Def}}
\newcommand{\sym}{\operatorname{sym}}
\newcommand{\tc}{\operatorname{tc}}
\newcommand{\rank}{\operatorname{rank}}
\newcommand{\namerank}{\operatorname{rank}_{\mathrm{name}}}
\newcommand{\Con}{\operatorname{Con}}
\newcommand{\forcesS}{\Vdash_{\mathrm{sym}}}
\newcommand{\restr}{\mathbin{\upharpoonright}}
\newcommand{\HOD}{\operatorname{HOD}}
\newcommand{\Col}{\operatorname{Col}}
\newcommand{\Coll}{\operatorname{Coll}}
\newcommand{\Aut}{\operatorname{Aut}}
\newcommand{\DC}{\mathsf{DC}}
\newcommand{\LM}{\mathsf{LM}}
\newcommand{\BP}{\mathsf{BP}}
\newcommand{\PSP}{\mathsf{PSP}}
\newcommand{\RR}{\mathsf{RR}}
\newcommand{\crit}{\operatorname{crit}}
\newcommand{\cf}{\operatorname{cf}}
\newcommand{\UE}{\mathsf{UE}}
\newcommand{\Izero}{\mathsf{I0}}

\theoremstyle{plain}
\newtheorem{theorem}{Theorem}[section]
\newtheorem{proposition}[theorem]{Proposition}
\newtheorem{lemma}[theorem]{Lemma}
\newtheorem{corollary}[theorem]{Corollary}
\newtheorem*{thmA}{Theorem A}
\newtheorem*{thmB}{Theorem B}
\theoremstyle{definition}
\newtheorem{definition}[theorem]{Definition}
\theoremstyle{remark}
\newtheorem{remark}[theorem]{Remark}
\newtheorem{question}[theorem]{Question}

\title[Vop\v{e}nka's Principle without Choice]{Vop\v{e}nka's Principle without Choice:\\ Preservation under Symmetric Extensions}
\author{Tom de Groot}
\author{Wojciech Aleksander Wo{\l}oszyn}
\date{}

\begin{document}

\begin{abstract}

We \revision{give a detailed proof} that every set-sized symmetric extension of a model of \(\ZF+\VP\) again satisfies \(\ZF+\VP\), where \(\VP\) is formulated for arbitrary set-sized languages.  For each standard \(n\geq1\), the same construction preserves \(\VP(\Pi_n)\), allowing set parameters and arbitrary set-sized languages.  The theories \(\ZF+\VP\) and \(\ZFC+\VP\) are equiconsistent, and \(\ZF+\VP\) proves a proper class of L\"owenheim--Skolem cardinals.  Relative to \(\Con(\ZF+\VP)\), the axiom \(\DC\) is independent of \(\ZF+\VP+\neg\AC\), and there are Feferman--L\'evy and full Solovay models satisfying \(\VP\).  \revision{We consider} whether every countable model of \(\ZF+\VP\) has a class-generic extension satisfying \(\ZFC+\VP\).

\end{abstract}

\maketitle

\section*{Acknowledgements}

\revision{We thank W.~Hugh Woodin for clarifying the attribution
of the equiconsistency result and explaining the original motivation
for the Choice-restoring methods.}

{The authors used OpenAI's ChatGPT 5.6 Sol and Anthropic's Claude Fable 5 for proof development, literature reviews, editorial revisions, and LaTeX preparation.  The authors retain full responsibility for the mathematical arguments and claims, the accuracy, and the final text.}

\section{Introduction}

Vop\v{e}nka's principle asserts that every proper class of structures in a common set-sized language contains distinct members $M$ and $N$ with an elementary embedding $j:M\to N$.  Throughout, $\VP$ denotes the corresponding first-order scheme for classes definable with set parameters.  For each standard $n\geq1$, $\VP(\Pi_n)$ restricts the defining formula to $\Pi_n$.

\revision{The equiconsistency in Theorem~B is an unpublished result
of Woodin that follows from his work in \emph{Suitable extender models I}
\cite{Woodin}. We give a detailed proof in
Section~\ref{sec:consistency-transfer}.}

Over $\ZFC$, Bagaria showed that $\VP$ is equivalent to the existence of a $C^{(n)}$-extendible cardinal for every $n$ \cite{BagariaCn}.  Tzouvaras proved that a general downward-absoluteness principle for $\VP$ implies $\AC$ and asked whether $\AC$ is independent of $\ZF+\VP$ \cite[Question~4.1]{Tzouvaras}.  The following two theorems answer this question.

\begin{thmA}
Every set-sized symmetric extension of a model of $\ZF+\VP$ satisfies $\ZF+\VP$.  For every standard $n\geq1$, every set-sized symmetric extension of a model of $\ZF+\VP(\Pi_n)$ satisfies $\ZF+\VP(\Pi_n)$.
\end{thmA}

\begin{thmB}[\revision{Woodin}]
The theories $\ZF+\VP$ and $\ZFC+\VP$ are equiconsistent.
\end{thmB}

Applied to the basic Cohen symmetric system, Theorem~A shows that $\Con(\ZF+\VP)$ implies $\Con(\ZF+\VP+\neg\AC)$, while Theorem~B shows that $\Con(\ZF+\VP)$ implies $\Con(\ZFC+\VP)$.  Hence $\AC$ is independent of $\ZF+\VP$, relative to $\Con(\ZF+\VP)$.

Brooke-Taylor proved over $\ZFC$ that set forcing preserves $\VP$ for structures in a fixed finite language \cite[Lemma~23 and Theorem~24]{BrookeTaylor}.  Hayut and Karagila developed criteria for lifting elementary embeddings through symmetric extensions \cite[Theorems~4.7 and~4.9]{HayutKaragila}.  Our proof constructs in the symmetric extension the graph of \revision{the restriction to the two structures}.

Instead of choosing one name at each rank, the proof uses the set of all names of least possible rank.  Applying $\VP$ to rank structures with these sets as distinguished constants gives an embedding that fixes the symmetric system.  Proposition~\ref{prop:symmetric-lift} then places the required restriction of its lift in the symmetric extension.

For Theorem~B, we first obtain a model with unbounded $C^{(n)}$-extendible cardinals for every standard $n$.  Woodin's forcing at a sufficiently correct choiceless supercompact cardinal gives a $\ZFC$ rank model satisfying any prescribed finite subset of $\ZFC+\VP$.  Compactness finishes the proof.

Sections~2 and~3 prove the lifting and preservation theorems.  Section~4 derives a proper class of L\"owenheim--Skolem cardinals, and Section~5 proves Theorem~B.  Sections~6--8 apply Theorem~A to the Cohen, Feferman--L\'evy, and full Solovay constructions.  Section~9 asks whether Choice can instead be restored by class forcing over a given countable model while preserving $\VP$.

All forcing notions have a largest condition $1_{\mathbb P}$.  We write $\namerank(\tau)$ for the ground-model rank of a forcing name $\tau$ and $f``X$ for the direct image of $X$ under $f$.  We code structures so that the universe of a structure is contained in its transitive closure.

\section{Symmetric models and the lifting lemma}

Let $\mathscr S=\langle\mathbb P,\Gamma,\mathcal F\rangle$ be a symmetric system, where $\Gamma$ is a group of automorphisms of $\mathbb P$ and $\mathcal F$ is a normal filter of subgroups of $\Gamma$.  We write $\pi p$ for the action of $\pi$ on a condition $p$ and $\pi\cdot\sigma$ for the induced recursive action on a name $\sigma$.  The latter is defined by
$\pi\cdot\sigma=\{\langle\pi\cdot\tau,\pi p\rangle:\langle\tau,p\rangle\in\sigma\}$.  Put $\sym(\sigma)=\{\pi\in\Gamma:\pi\cdot\sigma=\sigma\}$.  A name is symmetric if its stabilizer belongs to $\mathcal F$.  For a $\mathbb P$-name $\sigma$, let $\operatorname{cl}(\sigma)$ be the least set of names $X$ with $\sigma\in X$ such that $\upsilon\in X$ whenever $\langle\upsilon,p\rangle\in\tau$ for some $\tau\in X$.  A name $\sigma$ is hereditarily symmetric if
\begin{equation}\label{eq:HS-characterization}
\sym(\tau)\in\mathcal F\quad\text{for every }\tau\in\operatorname{cl}(\sigma).
\end{equation}
Let $\HS$ denote the class of hereditarily symmetric names.  Since $\sigma\in\operatorname{cl}(\sigma)$, every hereditarily symmetric name is symmetric.  We write $\forcesS$ for the symmetric forcing relation.  We use its definability, symmetry, and truth lemmas; see \cite[Chapter~5]{Jech}, \cite[Section~2]{Karagila}, and {\cite[Lemma~2.2 and Theorem~2.5]{HayutKaragilaSmall}}.  The action on names and the operation $\sigma\mapsto\operatorname{cl}(\sigma)$ are uniformly first-order definable by well-founded recursion.

We shall repeatedly use the following restriction lemma.  For a transitive set $a$, write $\operatorname{Sat}(a)$ for the canonical satisfaction relation of $(a,\in)$.

\begin{lemma}\label{lem:restriction-elementarity}
Let $j:M\to N$ be elementary between transitive sets, and let $a\in M$ be transitive.  If $\operatorname{Sat}(a)\in M$, then $j\restr a:(a,\in)\longrightarrow(j(a),\in)$ is elementary.
\end{lemma}

\begin{proof}
The satisfaction relation is uniquely defined by recursion on formulas, so
$j(\operatorname{Sat}(a))=\operatorname{Sat}(j(a))$.  Elementarity of $j$ now gives, for every formula $\varphi$ and every finite tuple $\bar x$ from $a$,
\[
(a,\in)\models\varphi(\bar x)
\quad\text{if and only if}\quad
(j(a),\in)\models\varphi(j(\bar x)).\qedhere
\]
\end{proof}

The usual symmetric-model theorem holds over a ground model of $\ZF$.  The Power Set and Collection arguments are recorded because they require care without Choice.

\begin{proposition}\label{prop:symmetric-ZF}
Suppose that $V\models\ZF$, that $\mathscr S$ is set-sized, and that $G\subseteq\mathbb P$ is $V$-generic.  Then
$W=\HS^G=\{\sigma^G:\sigma\in\HS\}$ is a transitive model of $\ZF$ with $V\subseteq W\subseteq V[G]$.
\end{proposition}

\begin{proof}
For $\sigma,\tau\in\HS$, define the intersection name
\[
\tau\cap^\bullet\sigma=\bigl\{\langle\rho,r\rangle:\exists s\,\bigl(\langle\rho,s\rangle\in\sigma, r\leq s,
\text{ and }r\forcesS\rho\in\tau\bigr)\bigr\}.
\]
Every $\pi\in\sym(\sigma)\cap\sym(\tau)$ fixes $\tau\cap^\bullet\sigma$, and every name appearing as the first coordinate of a pair in this name also appears as the first coordinate of a pair in $\sigma$.  Hence $\tau\cap^\bullet\sigma\in\HS$.  Moreover, $\tau\cap^\bullet\sigma\subseteq \operatorname{dom}(\sigma)\times\mathbb P$ and $(\tau\cap^\bullet\sigma)^G=\tau^G\cap\sigma^G$.  Let $B_\sigma=\HS\cap\mathcal P\bigl(\operatorname{dom}(\sigma)\times\mathbb P\bigr)$, which is a set by Separation, and define $\dot Q_\sigma=\bigl\{\langle u,p\rangle:u\in B_\sigma\text{ and }p\forcesS u\subseteq\sigma\bigr\}$.
If $\pi\in\sym(\sigma)$, then $\pi$ permutes $B_\sigma$, because it preserves $\operatorname{dom}(\sigma)\times\mathbb P$ and hereditary symmetry.  For $u\in B_\sigma$ and $p\in\mathbb P$,
\[
p\forcesS u\subseteq\sigma
\quad\text{if and only if}\quad
\pi p\forcesS\pi\cdot u\subseteq\sigma.
\]
Thus $\sym(\sigma)\leq\sym(\dot Q_\sigma)$.  Since $\sym(\sigma)\in\mathcal F$ and every name occurring directly in $\dot Q_\sigma$ belongs to $B_\sigma\subseteq\HS$, we have $\dot Q_\sigma\in\HS$.

If $x\in\dot Q_\sigma^G$, then $x\in W$ and $x\subseteq\sigma^G$.  Conversely, let $x\in W$ with $x\subseteq\sigma^G$, and choose $\tau\in\HS$ with $\tau^G=x$.  Then $\tau\cap^\bullet\sigma\in B_\sigma$ and $(\tau\cap^\bullet\sigma)^G=x$.  Since $(\tau\cap^\bullet\sigma)^G\subseteq\sigma^G$, there is $p\in G$ such that $p\forcesS\tau\cap^\bullet\sigma\subseteq\sigma$.  Thus $x\in\dot Q_\sigma^G$, and consequently $\dot Q_\sigma^G=\mathcal P(\sigma^G)\cap W$.
This proves Power Set in $W$.

To verify Collection, fix $\sigma,\vec u\in\HS$ and suppose $W\models\forall x\in\sigma^G\,\exists y\,\varphi(x,y,\vec u^G)$.  Choose $p\in G$ such that $p\forcesS\forall x\in\sigma\,\exists y\,\varphi(x,y,\vec u)$.  Let $I=\bigl\{\langle\rho,s,r\rangle:\langle\rho,s\rangle\in\sigma,\ r\leq p,\text{ and }r\leq s\bigr\}$.  If $\langle\rho,s,r\rangle\in I$, then $r\forcesS\rho\in\sigma$, so $r\forcesS\exists y\,\varphi(\rho,y,\vec u)$.  There are therefore $t\leq r$ and $\tau\in\HS$ such that $t\forcesS\varphi(\rho,\tau,\vec u)$.  Collection and Separation in $V$ give a set $C$ of quintuples $\langle\rho,s,r,t,\tau\rangle$ such that every $\langle\rho,s,r\rangle\in I$ occurs as the first three coordinates of a member of $C$, with $t\leq r$ and $t\forcesS\varphi(\rho,\tau,\vec u)$.

Let $H=\sym(\sigma)\cap\bigcap_i\sym(u_i)\in\mathcal F$, put $b_0=\bigl\{\langle\tau,t\rangle:\exists\rho,s,r\,\langle\rho,s,r,t,\tau\rangle\in C\bigr\}$, and set $b=\bigcup_{\pi\in H}\pi\cdot b_0$.
Replacement shows that $b$ is a set.  Every $\xi\in H$ fixes $b$, since left multiplication by $\xi$ permutes $H$.  Thus $H\leq\sym(b)$.  Each name occurring directly in $b$ has the form $\pi\cdot\tau$ with $\tau\in\HS$, and normality of $\mathcal F$ gives $\pi\cdot\tau\in\HS$.  Since $H\in\mathcal F$, we have $b\in\HS$.  The identity belongs to $H$, so $b_0\subseteq b$.

Let $x\in\sigma^G$.  Choose $\langle\rho,s\rangle\in\sigma$ with $s\in G$ and $\rho^G=x$, and choose $r_0\in G$ with $r_0\leq p,s$.  Define $D_{\rho,s}=\bigl\{t:\exists r,\tau\,\langle\rho,s,r,t,\tau\rangle\in C\bigr\}$.  If $q\leq r_0$, then $\langle\rho,s,q\rangle\in I$, so $C$ contains some $\langle\rho,s,q,t,\tau\rangle$ with $t\leq q$.  {Hence $D_{\rho,s}$ is dense below $r_0\in G$, and genericity gives $t\in G\cap D_{\rho,s}$.}  Choose $\langle\rho,s,r,t,\tau\rangle\in C$.  Then $\langle\tau,t\rangle\in b_0\subseteq b$, so $\tau^G\in b^G$ and $W\models\varphi(x,\tau^G,\vec u^G)$.

Power Set and Collection are now established.  For the remaining axioms, transitivity gives Extensionality and Foundation, canonical names give Empty Set, Pairing, Union, and Infinity, and symmetric subnames give Separation.  Collection and Separation imply Replacement.

Thus $W\models\ZF$.
\end{proof}
\begin{lemma}\label{lem:HS-absoluteness}
Suppose that $V_\theta\models\ZF$ and $\tc(\{\mathscr S\})\subseteq V_\theta$.  If $\sigma\in V_\theta$ is a $\mathbb P$-name, then hereditary symmetry computed in $V_\theta$ agrees with hereditary symmetry computed in $V$.  Consequently, if $V_\eta\models\ZF$ and $j:V_\theta\to V_\eta$ is elementary with $j(\mathscr S)=\mathscr S$, then $j(\sigma)\in\HS$ whenever $\sigma\in\HS\cap V_\theta$.
\end{lemma}

\begin{proof}
Well-founded recursion is absolute to the transitive model $V_\theta$, so the action on names and $\operatorname{cl}(\tau)$ are computed correctly there.  Thus $\operatorname{cl}(\tau)\in V_\theta$.  For every $\nu\in\operatorname{cl}(\tau)$, $V_\theta$ and $V$ compute $\sym(\nu)$ from the same group $\Gamma$ and agree whether it belongs to $\mathcal F$.  Hence they agree on~\eqref{eq:HS-characterization}.

Now suppose that $j:V_\theta\to V_\eta$ is elementary, $j(\mathscr S)=\mathscr S$, and $\sigma\in\HS\cap V_\theta$.  By the first part, $V_\theta$ regards $\sigma$ as hereditarily symmetric, so elementarity gives that $V_\eta$ regards $j(\sigma)$ as hereditarily symmetric with respect to $\mathscr S$.

Since $j(\mathscr S)=\mathscr S\in V_\eta$ and $V_\eta$ is transitive, $\tc(\{\mathscr S\})\subseteq V_\eta$.  Applying the first part to $V_\eta$ yields $j(\sigma)\in\HS$ in $V$.
\end{proof}

For names $\sigma$ and $\tau$, let $\langle\sigma,\tau\rangle^\bullet$ be the canonical name for their Kuratowski ordered pair.  The construction is equivariant:
\[
\pi\cdot\langle\sigma,\tau\rangle^\bullet
=\langle\pi\cdot\sigma,\pi\cdot\tau\rangle^\bullet.
\]
{The canonical ordered-pair construction also preserves hereditary symmetry: if $\sigma,\tau\in\HS$, every new name in its Kuratowski code is fixed by $\sym(\sigma)\cap\sym(\tau)$; hence $\langle\sigma,\tau\rangle^\bullet\in\HS$.}

\begin{proposition}\label{prop:symmetric-lift}
Suppose that $V_\theta,V_\eta\models\ZF$, and let $j:V_\theta\to V_\eta$ be elementary.  Suppose that $j(z)=z$ for every $z\in\tc(\{\mathscr S\})$.  If $G\subseteq\mathbb P$ is $V$-generic and $W=\HS^G$, then there is a function $E_j\in W$ such that
\[
E_j(\sigma^G)=j(\sigma)^G\qquad(\sigma\in\HS\cap V_\theta).
\]
Thus, in $V[G]$, $E_j=j_G\restr\{\sigma^G:\sigma\in\HS\cap V_\theta\}$, where $j_G:V_\theta[G]\to V_\eta[G]$ is the ordinary forcing lift of $j$, defined by $j_G(\sigma^G)=j(\sigma)^G$.
\end{proposition}

\begin{proof}
Since $j$ fixes every member of $\tc(\{\mathscr S\})$, it fixes $\mathbb P$ pointwise, so $j``G=G$.  Since $G$ is $V$-generic, it is generic over both $V_\theta$ and $V_\eta$.

{Define $j_G(\sigma^G)=j(\sigma)^G$ for every $\mathbb P$-name $\sigma\in V_\theta$.  The forcing theorem and elementarity show that this is well defined and that}
\[
p\Vdash_{V_\theta}\varphi(\vec\sigma)
\quad\text{if and only if}\quad
p\Vdash_{V_\eta}\varphi(j(\vec\sigma))
\]
{for every formula $\varphi$, condition $p\in\mathbb P$, and finite tuple $\vec\sigma$ of names from $V_\theta$.  Thus $j_G:V_\theta[G]\to V_\eta[G]$ is elementary.}
Put $D_\theta=\HS\cap V_\theta$.  Hereditary symmetry is first-order definable from $\mathscr S$, so $D_\theta$ is a set by Separation.  Lemma~\ref{lem:HS-absoluteness} gives $j(\sigma)\in\HS$ for every $\sigma\in D_\theta$.

The action of $\Gamma$ on $\mathbb P$-names is defined by well-founded recursion.  Hence, for $\pi\in\Gamma$ and every $\mathbb P$-name $\sigma\in V_\theta$, elementarity and $j(\pi)=\pi$ give
\begin{equation}\label{eq:equivariance}
j(\pi\cdot\sigma)
=j(\pi)\cdot j(\sigma)
=\pi\cdot j(\sigma).
\end{equation}
Moreover, $\operatorname{cl}(\pi\cdot\sigma)=\pi``\operatorname{cl}(\sigma)$, and $\sym(\pi\cdot\nu)=\pi\sym(\nu)\pi^{-1}$.  Since $\mathcal F$ is normal, $\sigma$ is hereditarily symmetric if and only if $\pi\cdot\sigma$ is hereditarily symmetric.  As $\pi$ bijects the $\mathbb P$-names in $V_\theta$, it follows that $\pi``D_\theta=D_\theta$.

The required graph name is
\[
\dot E_j=\bigl\{\langle\langle\sigma,j(\sigma)\rangle^\bullet,1_{\mathbb P}\rangle:\sigma\in D_\theta\bigr\}.
\]
Since $D_\theta$ is a set, $\dot E_j$ is a $\mathbb P$-name, and every pair name occurring in it belongs to $\HS$.  For each $\pi\in\Gamma$, equivariance of the ordered-pair name, $\pi 1_{\mathbb P}=1_{\mathbb P}$, equation~\eqref{eq:equivariance}, and $\pi``D_\theta=D_\theta$ give
\[
\pi\cdot\dot E_j
 =\bigl\{\langle\langle\pi\cdot\sigma,j(\pi\cdot\sigma)\rangle^\bullet,1_{\mathbb P}\rangle:\sigma\in D_\theta\bigr\}
 =\dot E_j.
\]
Thus $\sym(\dot E_j)=\Gamma\in\mathcal F$, so $\dot E_j\in\HS$.

Let $E_j=\dot E_j^G$.  Then
\[
E_j=\bigl\{\langle\sigma^G,j(\sigma)^G\rangle:\sigma\in D_\theta\bigr\}.
\]
If $\sigma^G=\tau^G$, the well-definedness of $j_G$ gives $j(\sigma)^G=j(\tau)^G$.  Hence $E_j$ is a function in $W$, with
\[
E_j(\sigma^G)=j(\sigma)^G
\qquad
(\sigma\in D_\theta).\qedhere
\]
\end{proof}

\section{Preservation}

\begin{lemma}\label{lem:ZF-rank-criterion}
Let $\theta>\omega$ be a limit ordinal.  If no function $f:a\to\theta$ with $a\in V_\theta$ has cofinal range, then $V_\theta\models\ZF$.  The hypothesis on $\theta$ is $\Pi_1$.
\end{lemma}

\begin{proof}
Only Replacement and Collection require proof.  Suppose $a,p\in V_\theta$ and $V_\theta\models\forall x\in a\,\exists!y\,\varphi(x,y,p)$.  Replacement in $V$ gives the corresponding function $F$.  If the ranks of its values were unbounded in $\theta$, then $x\mapsto\rank(F(x))$ would be cofinal in $\theta$, contrary to the hypothesis.  Hence $F``a\in V_\theta$.

The same rank argument proves Collection.  Suppose $V_\theta\models\forall x\in a\,\exists y\,\varphi(x,y,p)$.  For each $x\in a$, take the least rank of a suitable $y\in V_\theta$.  These ranks are bounded below $\theta$.  If $\beta<\theta$ lies above them, Separation applied to $V_\beta$ gives a set containing a suitable $y$ for every $x\in a$.  Finally, $x\in V_\theta$ is $\Delta_1$ in $\theta$, and ``$f:a\to\theta$ is cofinal'' is $\Delta_0$ in $f,a,\theta$.  Thus the hypothesis is $\Pi_1$.
\end{proof}

\begin{lemma}\label{lem:ZF-ranks}

Assume $\ZF+\VP$.  Above every ordinal $\gamma$ there is a limit ordinal $\kappa$ such that no map from any $a\in V_\kappa$ is cofinal in $\kappa$.  Consequently $V_\kappa\models\ZF$.  Under $\AC$, such a $\kappa$ may be chosen measurable.
\end{lemma}

\begin{proof}

Fix $\gamma$ and put $\bar\gamma=\max\{\gamma,\omega\}$.  For $\lambda>\bar\gamma$, consider
\[
\mathcal B_\lambda=\bigl\langle V_{\lambda+\omega},\in,\lambda,
(z)_{z\in V_{\bar\gamma+1}}\bigr\rangle
\]
 in the set-sized language with membership, one distinguished constant, and constants for $V_{\bar\gamma+1}$.  The structures $\mathcal B_\lambda$ form a definable proper class.  Applying $\VP$ to this class gives distinct ordinals $\lambda,\mu$ and an elementary embedding $j:\mathcal B_\lambda\longrightarrow\mathcal B_\mu$.
By elementarity and the choice of constants, $j(\lambda)=\mu$, while $j(z)=z$ for every $z\in V_{\bar\gamma+1}$.  Since $\lambda\ne\mu$, the embedding moves an ordinal.  Let $\kappa$ be the least ordinal moved by $j$.  Then $\bar\gamma<\kappa\leq\lambda$, $j(\kappa)>\kappa$, and a rank induction shows that $j(x)=x$ for every $x\in V_\kappa$.

 Suppose that $f:a\to\kappa$ is cofinal for some $a\in V_\kappa$.  The embedding fixes $a$ and each of its elements, so $j(f):a\to j(\kappa)$ is cofinal and $j(f)(x)=f(x)<\kappa$ for every $x\in a$.  This is impossible because $j(\kappa)>\kappa$.  Thus no map from a member of $V_\kappa$ is cofinal in $\kappa$.  The ordinal $\kappa$ cannot be a successor, since $j$ fixes every smaller ordinal.  As $\kappa>\bar\gamma\geq\omega$, Lemma~\ref{lem:ZF-rank-criterion} applies, so $V_\kappa\models\ZF$.

Assume now that $\AC$ holds and define $U=\{X\subseteq\kappa:\kappa\in j(X)\}$.

Elementarity shows that $U$ is an ultrafilter on $\kappa$.  Because $j$ fixes every ordinal below $\kappa$, $U$ is nonprincipal.  Moreover, $U$ is $\kappa$-complete, since $V_{\lambda+\omega}$ contains every sequence of fewer than $\kappa$ subsets of $\kappa$.  Under Choice, $\kappa$ must be a cardinal.  If it were not, then $\kappa$ could be partitioned into fewer than $\kappa$ singletons, contradicting the $\kappa$-completeness of the nonprincipal ultrafilter $U$.  Thus $U$ is a nonprincipal $\kappa$-complete ultrafilter on the cardinal $\kappa$, and therefore $\kappa$ is measurable.
\end{proof}

\begin{theorem}\label{thm:preservation}

Let $V\models\ZF+\VP$, with $\VP$ understood in the arbitrary-language sense fixed in the Introduction.  Let $\mathscr S=\langle\mathbb P,\Gamma,\mathcal F\rangle\in V$ be a set-sized symmetric system, and let $G\subseteq\mathbb P$ be $V$-generic.  Then
\[
W=\HS^G\models\ZF+\VP.
\]
\end{theorem}

\begin{proof}

Proposition~\ref{prop:symmetric-ZF} gives $W\models\ZF$.  Fix a formula $\varphi(x,y)$, a parameter $a\in W$, and a set-sized language $L\in W$ such that
$\mathcal A=\{M\in W:W\models\varphi(M,a)\}$ is a proper class of $L$-structures.  Choose $\dot a,\dot L\in\HS$ with $\dot a^G=a$ and $\dot L^G=L$.  For $p\in\mathbb P$, let $R_p$ be the class of ordinals $\alpha$ for which there is a name $\tau\in\HS$ satisfying
\begin{equation}\label{eq:rank-witness}
p\forcesS ``\tau\text{ is an }\dot L\text{-structure},\ \varphi(\tau,\dot a),
\text{ and }\rank(\tau)=\check\alpha.''
\end{equation}

{In \eqref{eq:rank-witness}, $\rank(\tau)$ denotes the rank of the value of $\tau$ in the generic extension.  By contrast, $\namerank(\tau)$ is the von Neumann rank of the name $\tau$ itself.}

There is a condition $q\in G$ for which $R_q$ is unbounded in $\Ord$.  Suppose not.  Let $D_{\mathrm{bd}}=\{p\in\mathbb P:R_p\text{ is bounded}\}$ and, for $p\in D_{\mathrm{bd}}$, let $b(p)$ be the least ordinal such that $R_p\subseteq b(p)$.  The domain $D_{\mathrm{bd}}$ is a subset of the set $\mathbb P$, and definability of symmetric forcing together with Replacement makes $b$ a set function in $V$.  Our supposition gives $G\subseteq D_{\mathrm{bd}}$.

In $V[G]$, set $\delta=\sup b``G$.  This is an ordinal, so preservation of ordinals gives $\delta\in V\subseteq W$.  If $M\in\mathcal A$, choose $\tau\in\HS$ with $\tau^G=M$ and let $\alpha=\rank(M)$.  By the symmetric truth lemma, some $p\in G$ forces the assertion in \eqref{eq:rank-witness}; hence $\alpha<b(p)\leq\delta$.  Thus $\mathcal A\subseteq V_\delta^W$, and Separation in $W$ makes $\mathcal A$ a set, a contradiction.

For $\alpha\in R_q$, let $\Psi_\alpha(\tau)$ assert that $\tau\in\HS$ and that \eqref{eq:rank-witness} holds with $q$ in place of $p$.  We use the following definitions:
\[
\nu_\alpha=\min\{\namerank(\tau):\Psi_\alpha(\tau)\},
\qquad
X_\alpha=\{\tau\in V_{\nu_\alpha+1}:\namerank(\tau)=\nu_\alpha
\text{ and }\Psi_\alpha(\tau)\}.
\]
{The ordinal $\nu_\alpha$ exists, $X_\alpha$ is a nonempty set by Separation, and both are uniformly definable from $\alpha$ and the fixed parameters.}

Choose $\rho>\omega$ such that $\tc(\{\mathscr S,\dot a,\dot L\})\subseteq V_\rho$.  By Lemma~\ref{lem:ZF-ranks}, for each $\alpha\in R_q$ there are arbitrarily high $\theta$ with $V_\theta\models\ZF$.  The relevant rank is
\[
\theta_\alpha=\min\bigl\{\theta:V_\theta\models\ZF\text{ and }
\theta>\max\{\rho,\alpha,\rank(X_\alpha)\}+\omega\bigr\}.
\]
In the common language consisting of a binary relation symbol, two distinguished constant symbols, and constant symbols $c_z$ for $z\in V_\rho$, use the structure
\[
\mathcal M_\alpha=
\bigl\langle V_{\theta_\alpha},\in,\alpha,X_\alpha,(z)_{z\in V_\rho}\bigr\rangle
\qquad(\alpha\in R_q).
\]

The structures $\mathcal M_\alpha$ form a definable proper class since $R_q$ is unbounded.  By $\VP$, there is an elementary embedding $j:\mathcal M_\alpha\to\mathcal M_\beta$ with $\alpha\ne\beta$.  The distinguished constants give $j(\alpha)=\beta$, $j(X_\alpha)=X_\beta$, and $j(z)=z$ for every $z\in V_\rho$.

Fix $\tau\in X_\alpha$.  Then $j(\tau)\in X_\beta$, and the definitions of $X_\alpha$ and $X_\beta$ give
\begin{equation}\label{eq:q-witness}
q\forcesS ``\tau\text{ and }j(\tau)\text{ are }\dot L\text{-structures satisfying }
\varphi(-,\dot a)\text{ with ranks }\check\alpha\text{ and }\check\beta.''
\end{equation}
Since $j$ fixes $V_\rho$ pointwise, it fixes every member of $\tc(\{\mathscr S,\dot a,\dot L\})$.

Apply Proposition~\ref{prop:symmetric-lift} to obtain the elementary lift
$j_G:V_{\theta_\alpha}[G]\to V_{\theta_\beta}[G]$ and a function $E_j\in W$ agreeing with $j_G$ on values of names in $\HS\cap V_{\theta_\alpha}$.  Put $M=\tau^G$ and $N=j(\tau)^G$.  Then $j_G(M)=N$.

If $s\in L=\dot L^G$, then $s=\sigma^G$ for some $\langle\sigma,p\rangle\in\dot L$ with $p\in G$.  Since $\sigma\in\operatorname{cl}(\dot L)\subseteq V_\rho$, we have $j(\sigma)=\sigma$, and hence $j_G(s)=s$.  Thus $j_G$ fixes $L$ pointwise.  Moreover, $j(\dot L)=\dot L$ and $j_G(L)=L$.

By the coding convention of the Introduction, $\lvert M\rvert\subseteq\tc(M)$.  An induction on name evaluation shows that every element of $\tc(M)$ is $\sigma^G$ for some $\sigma\in\operatorname{cl}(\tau)$.  Since $\tau\in\HS\cap V_{\theta_\alpha}$, every such $\sigma$ belongs to $\HS\cap V_{\theta_\alpha}$.  Hence $|M|\subseteq\operatorname{dom}(E_j)$, and $e=E_j\restr|M|$ belongs to $W$.  In $V[G]$, the map $e$ is $j_G\restr|M|$.  Since the universe of a structure is uniformly definable from its fixed set code,
\[
j_G(|M|)=|j_G(M)|=|N|,
\qquad j_G``|M|\subseteq|N|.
\]

{Satisfaction for set-sized structures is absolute between the transitive models under consideration.  Since $j_G$ fixes $L$ pointwise, elementarity gives}
\[
M\models\psi(\bar m)\quad\text{if and only if}\quad
N\models\psi(e(\bar m))
\]
{for every $L$-formula $\psi$ and every finite tuple $\bar m$ from $|M|$.  Thus $W$ regards $e$ as an $L$-elementary embedding from $M$ to $N$.}

Finally, $q\in G$ and \eqref{eq:q-witness} imply that $M,N\in\mathcal A$, with ranks $\alpha$ and $\beta$.  Thus $M\neq N$.  This proves the arbitrary instance of $\VP$, so $W\models\ZF+\VP$.
\end{proof}

\begingroup

\begin{corollary}\label{cor:Pi-preservation}
Let $n\geq1$ be standard.  If $V\models\ZF+\VP(\Pi_n)$, $\mathscr S=\langle\mathbb P,\Gamma,\mathcal F\rangle\in V$ is a set-sized symmetric system, $G\subseteq\mathbb P$ is $V$-generic, and $W=\HS^G$, then
\[
W\models\ZF+\VP(\Pi_n).
\]
\end{corollary}

\begin{proof}
{Symmetric forcing preserves L\'evy complexity: for each standard $k$, forcing for a $\Sigma_k$ formula is $\Sigma_k$, and forcing for a $\Pi_k$ formula is $\Pi_k$, uniformly in the formula code.  For bounded formulas this follows by induction on formulas and name rank; atomic clauses lower name rank, and bounded quantifiers range over the domains of names.  At an existential step, Collection over $\{r\in\mathbb P:r\leq p\}$ records names on a dense set below $p$.  Separation then yields the required forcing relation at the same complexity.}

We use the following elementary coding observation.  If $S(z)$ is $\Sigma_k$, $k\geq1$, write it as $\exists u\,\psi(z,u)$ with $\psi$ in $\Pi_{k-1}$.  When $S(z)$ holds, let $\xi$ be the least rank of such a $u$ and let $C$ be the set of all such $u$ of rank $\xi$.  There is a uniform $\Pi_k$ formula $\chi_S(z,c)$ saying that
\[
c=\langle\xi,V_\xi,V_{\xi+1},C\rangle
\]
has these properties.  Indeed, once the two rank segments are included in $c$, all quantifiers used for minimality and for the exact description of $C$ are bounded; the equations defining the rank segments are $\Pi_1$.  Hence $c$ is unique and $\chi_S(z,c)$ implies $S(z)$.

The $\Pi_n$ argument uses this observation twice: to record the least name rank together with all witnesses at that rank, and to record the least rank satisfying Lemma~\ref{lem:ZF-rank-criterion}.

{Lemma~\ref{lem:ZF-ranks} gives unboundedly many ordinals satisfying the $\Pi_1$ condition of Lemma~\ref{lem:ZF-rank-criterion}.  For $b<\theta$, the statement that no ordinal $\eta$ with $b<\eta<\theta$ satisfies this condition is $\Sigma_1$.  By the preceding observation, this statement can be expressed by a $\Pi_1$ formula involving an additional set $d$, uniquely determined by $b$ and $\theta$.  Thus the pair $(\theta,d)$, where $\theta$ is the least such ordinal above $b$, is $\Pi_1$-definable from $b$.}  That set and the satisfaction relation of $V_\theta$ lie in $V_{\theta+\omega}$.

Now let a $\Pi_n$ formula $\varphi(x,a)$ define a proper class $\mathcal A$ of structures in a set-sized language $L$ in $W$.  Choose names $\dot a,\dot L\in\HS$.  Define $R_p$ as in the proof of Theorem~\ref{thm:preservation}.  The same boundedness argument gives $q\in G$ for which $R_q$ is unbounded.  Let $\Psi(\alpha,\tau)$ say that $\tau\in\HS$ and that $q$ forces $\tau$ to be an $\dot L$-structure of rank $\alpha$ satisfying $\varphi(\tau,\dot a)$.  By the first paragraph, $\Psi$ is $\Pi_n$.

For $\alpha\in R_q$, let $\nu_\alpha$ be the least von Neumann rank of a name $\tau$ satisfying $\Psi(\alpha,\tau)$.  We use the following notation:
\[
U_\alpha=V_{\nu_\alpha},\qquad U_\alpha^+=V_{\nu_\alpha+1},\qquad
X_\alpha=\{\tau\in U_\alpha^+:\namerank(\tau)=\nu_\alpha\text{ and }\Psi(\alpha,\tau)\}.
\]
The assertions
\[
\forall\tau\in U_\alpha\,\neg\Psi(\alpha,\tau)
\quad\text{and}\quad
\forall\tau\in U_\alpha^+\bigl((\namerank(\tau)=\nu_\alpha\wedge\Psi(\alpha,\tau))\longrightarrow\tau\in X_\alpha\bigr)
\]
are $\Sigma_n$.  Apply the preceding coding observation separately to them, and let $c_\alpha$ be the ordered pair of the resulting codes.  Its first component records that no name of rank below $\nu_\alpha$ satisfies $\Psi(\alpha,\cdot)$; its second records that every name of rank $\nu_\alpha$ satisfying $\Psi(\alpha,\cdot)$ belongs to $X_\alpha$.  Keeping $c_\alpha$ as a distinguished component lets the defining formula state both coded $\Pi_n$ conditions without another existential quantifier.  Together with the rank-segment equations and the assertion that $X_\alpha$ is the nonempty set of rank-$\nu_\alpha$ names satisfying $\Psi(\alpha,\cdot)$, these clauses give a $\Pi_n$ definition from $\alpha$ of
\[
\langle\nu_\alpha,U_\alpha,U_\alpha^+,X_\alpha,c_\alpha\rangle.
\]

Choose $\rho>\omega$ with $\tc(\{\mathscr S,\dot a,\dot L\})\subseteq V_\rho$.  Let $\theta_\alpha$ be the least limit ordinal above $\max\{\rho,\allowbreak\alpha,\allowbreak\rank(\langle\nu_\alpha,\allowbreak U_\alpha,\allowbreak U_\alpha^+,\allowbreak X_\alpha,\allowbreak c_\alpha\rangle)\}+\omega$ admitting no cofinal map $f:a\to\theta_\alpha$ with $a\in V_{\theta_\alpha}$, and let $d_\alpha$ be the unique code witnessing this leastness in the $\Pi_1$ definition described above.  In the set-sized language with membership, finitely many distinguished constants, and constants for $V_\rho$, let
\[
\mathcal N_\alpha=\bigl\langle V_{\theta_\alpha+\omega},\in,
\alpha,\nu_\alpha,U_\alpha,U_\alpha^+,X_\alpha,c_\alpha,
\theta_\alpha,d_\alpha,(z)_{z\in V_\rho}\bigr\rangle.
\]
The rank-segment equations and the definition of $\theta_\alpha$ are $\Pi_1$, the least-name-rank clause is $\Pi_n$, and the other clauses are bounded or $\Delta_1$.  Thus $\{\mathcal N_\alpha:\alpha\in R_q\}$ is a $\Pi_n$-definable proper class: $R_q$ is unbounded, and the tuple is unique for each $\alpha$.

Apply $\VP(\Pi_n)$ in $V$.  There are distinct $\mathcal N_\alpha,\mathcal N_\beta$ and an elementary embedding $j:\mathcal N_\alpha\to\mathcal N_\beta$.  Uniqueness gives $\alpha\ne\beta$, and the distinguished constants give $j(\alpha)=\beta$, $j(X_\alpha)=X_\beta$, and $j(\theta_\alpha)=\theta_\beta$.  {Since $\operatorname{Sat}(V_{\theta_\alpha})\in V_{\theta_\alpha+\omega}$, Lemma~\ref{lem:restriction-elementarity} shows that $j_0=j\restr V_{\theta_\alpha}:V_{\theta_\alpha}\longrightarrow V_{\theta_\beta}$ is elementary.}  Both ranks satisfy $\ZF$, and the constants from $V_\rho$ make $j_0$ fix the symmetric system, every condition, and the names $\dot a$ and $\dot L$.

{Choose $\tau\in X_\alpha$.}  Then $j_0(\tau)\in X_\beta$, and $q$ forces that the two names denote members of $\mathcal A$ of ranks $\alpha$ and $\beta$.  Proposition~\ref{prop:symmetric-lift} lifts $j_0$ through $G$.  The restriction of the lifted map to the universe of $\tau^G$ lies in $W$.  {The lift fixes $L$ pointwise, so this restriction is an elementary $L$-embedding from $\tau^G$ to $j_0(\tau)^G$.}  Since $q\in G$ and $\alpha\ne\beta$, these are distinct members of $\mathcal A$.
\end{proof}

This corollary, together with Theorem~\ref{thm:preservation}, proves Theorem~A.

\endgroup

\begingroup

\section{L\"owenheim--Skolem cardinals}\label{sec:LS}

We use Usuba's definitions of weakly LS and LS cardinals \cite[Definition~1.1]{Usuba}.

\begin{definition}\label{def:LS}
Let $\kappa$ be an uncountable cardinal.
\begin{enumerate}
\item $\kappa$ is \emph{weakly L\"owenheim--Skolem}, or \emph{weakly LS}, if for every $\gamma<\kappa$, every $\alpha\geq\kappa$, and every $x\in V_\alpha$, there is $X\prec V_\alpha$ such that $V_\gamma\subseteq X$, $x\in X$, and the transitive collapse of $X$ belongs to $V_\kappa$.
\item $\kappa$ is \emph{L\"owenheim--Skolem}, or \emph{LS}, if for every $\gamma<\kappa$, every $\alpha\geq\kappa$, and every $x\in V_\alpha$, there are $\beta\geq\alpha$ and $X\prec V_\beta$ such that $V_\gamma\subseteq X$, $x\in X$, the transitive collapse of $X$ belongs to $V_\kappa$, and
\[
{}^{V_\gamma}(X\cap V_\alpha)\subseteq X.
\]
The last clause means that every function from $V_\gamma$ into $X\cap V_\alpha$ belongs to $X$.
\end{enumerate}
\end{definition}

\begin{lemma}\label{lem:limits-LS}
A limit of LS cardinals is LS, and a limit of weakly LS cardinals is weakly LS.
\end{lemma}

\begin{proof}
\begingroup

Suppose that $\kappa$ is a limit of LS cardinals.  It is uncountable, and it is a cardinal: if $b:\kappa\to\alpha$ were a bijection for some $\alpha<\kappa$, an LS cardinal $\mu$ with $\alpha<\mu<\kappa$ would make $b\restr\mu$ an injection from the initial ordinal $\mu$ into the smaller ordinal $\alpha$, a contradiction.  Now fix $\gamma<\kappa\leq\alpha$ and $x\in V_\alpha$.  Choose an LS cardinal $\mu$ with $\gamma<\mu<\kappa$, and apply its LS property to $(\gamma,\alpha,x)$.  The submodel supplied by the LS property of $\mu$ has transitive collapse in $V_\mu\subseteq V_\kappa$, so it also witnesses the LS property for $\kappa$.  The weakly LS argument is the same.
\endgroup
\end{proof}

For $n\geq1$, put $C^{(n)}=\{\alpha:V_\alpha\prec_{\Sigma_n}V\}$.

\revision{Mohammd's characterizations \cite[Theorems~3.1 and~2.8]{Mohammd} imply the following theorem.}

\begin{theorem}\label{thm:Vopenka-LS}
The theory $\ZF+\VP$ proves that there is a proper class of LS cardinals.
\end{theorem}

\begin{proof}
$\VP$ includes $\VP(\Pi_2)$.  By \cite[Theorem~3.1]{Mohammd}, there is a proper class of $1$-choiceless extendible ordinals.  Each is a cardinal by the observation following \cite[Definition~2.1]{Mohammd}, and \cite[Theorem~2.8]{Mohammd} identifies them with the $2$-choiceless supercompact cardinals.  It is therefore enough to prove that every $2$-choiceless supercompact cardinal is LS.

 Let $\kappa$ be $2$-choiceless supercompact, and fix $\delta<\kappa\leq\alpha$ and $x\in V_\alpha$.  Choose $\lambda\in C^{(2)}$ above $\alpha+\omega$.  Apply \cite[Definition~2.3]{Mohammd} with lower bound $\delta$ and parameter $a=\langle\alpha,x,\alpha+\omega\rangle$.  We obtain $\bar\lambda<\kappa$ in $C^{(2)}$, an element $\bar a\in V_{\bar\lambda}$, and an elementary embedding $j:V_{\bar\lambda}\longrightarrow V_\lambda$ such that $\crit(j)>\delta$ and $j(\bar a)=a$.  Decoding $\bar a$ gives ordinals $\bar\alpha<\bar\lambda$ and $\bar x\in V_{\bar\alpha}$ with $\bar a=\langle\bar\alpha,\bar x,\bar\alpha+\omega\rangle$, $j(\bar\alpha)=\alpha$, $j(\bar x)=x$, and $\bar\alpha+\omega<\bar\lambda$.  Moreover $\bar\alpha>\delta$, because $j$ fixes every ordinal below $\crit(j)$ whereas $j(\bar\alpha)=\alpha\geq\kappa>\delta$.

Put $X=j``V_{\bar\lambda}$.  Then $X\prec V_\lambda$, $V_\delta\subseteq X$, $x\in X$, and the transitive collapse of $X$ is $V_{\bar\lambda}\in V_\kappa$.  Moreover, $X\cap V_\alpha=j``V_{\bar\alpha}$.  If $f:V_\delta\to X\cap V_\alpha$, let $\bar f(u)$ be the unique $y\in V_{\bar\alpha}$ such that $j(y)=f(u)$.  Replacement gives $\bar f$, and $\bar\alpha+\omega<\bar\lambda$ implies $\bar f\in V_{\bar\lambda}$.  Since $j$ fixes $V_\delta$ pointwise, $j(\bar f)=f$.  Thus ${}^{V_\delta}(X\cap V_\alpha)\subseteq X$, and $\kappa$ is LS.
\end{proof}

\begin{corollary}\label{cor:singular-LS}
Above every ordinal there is a singular LS cardinal of cofinality $\omega$.
\end{corollary}

\begin{proof}
Starting above the given ordinal, recursively take the least LS cardinal above the preceding one and let $\kappa$ be the supremum of the resulting strictly increasing $\omega$-sequence.  Lemma~\ref{lem:limits-LS} makes $\kappa$ LS, and the sequence witnesses $\cf(\kappa)=\omega$.
\end{proof}

\begin{corollary}\label{cor:class-force-ZFC}
The theory $\ZF+\VP$ proves that there is a definable class forcing which forces $\ZFC$.
\end{corollary}

\begin{proof}
Combine Theorem~\ref{thm:Vopenka-LS} with Usuba's Choice-restoration theorem \cite[Corollary~4.8]{Usuba}; the resulting class forcing forces $\ZFC$.
\end{proof}
\endgroup

\begingroup

\section{Consistency transfer to Choice}\label{sec:consistency-transfer}

\revision{Woodin traces the Choice-restoring arguments used here
to his early work introducing $\Izero$ and proving determinacy in
$L(\mathbb R)$. He needed to obtain $\Izero$ in a universe satisfying
Choice to establish it as a large-cardinal hypothesis over $\ZFC$.
Starting from $\ZF$ with a super-Reinhardt cardinal, he obtained
an intermediate setting
satisfying $\lambda$-$\DC$ with an elementary embedding $j:V\to V$,
where $\lambda=\kappa_\omega(j)
=\sup_{n<\omega}j^n\bigl(\crit(j)\bigr)$.
From this setting he then obtained a model of
$\ZFC+\Izero$.\footnote{%
\revision{W.~H.~Woodin, personal communication, 7~October 2026.}}}

Proposition~\ref{prop:OH-pruning} reduces the forward consistency implication to a model with unbounded $C^{(n)}$-extendible cardinals for every standard $n$.  If this model satisfies Choice, it already satisfies $\ZFC+\VP$.  Otherwise, for each finite $F\subseteq\ZFC+\VP$, Theorem~\ref{thm:finite-restoration} produces a $\ZFC$ rank model satisfying $F$.  Compactness then yields a model of $\ZFC+\VP$.

\subsection*{{Choiceless extendibility}}

Following \cite[Definition~6.6]{Mohammd}, let $E_n(\kappa)$ assert that $\kappa$ is a cardinal and that, whenever $\lambda>\kappa$ belongs to $C^{(n)}$, there are $\theta>\lambda$ in $C^{(n)}$ and an elementary embedding $j:V_\lambda\to V_\theta$ with $\crit(j)=\kappa$ and $j(\kappa)>\lambda$.  We call such a $\kappa$ $C^{(n)}$-extendible; over $\ZFC$, this agrees with Bagaria's notion.  Let $\UE$, for unbounded extendibility, be the recursive scheme asserting that the $E_n$-cardinals are unbounded for every $n\geq1$.
This fixed-critical-point notion differs from the variable-critical-point supercompactness used in Theorem~\ref{thm:Vopenka-LS}.

\begin{lemma}\label{lem:OH-VP}
Over $\ZF$, $\UE$ implies $\VP$.
\end{lemma}

\begin{proof}
 Let $\mathcal A=\{M:\varphi(M,a)\}$ be a proper class of structures in a set-sized language $L$, and put $\rho=\rank(\tc(\{a,L\}))$.  Choose $n$ large enough that $C^{(n)}$-correct ranks are correct about $\varphi$, the coding of $L$-structures, and satisfaction for set-sized $L$-structures.  Choose a $C^{(n)}$-extendible cardinal $\kappa>\rho$ and $M\in\mathcal A$ with $\alpha=\rank(M)>\kappa$.  Let $\lambda\in C^{(n)}$ lie above $\alpha+\omega$, and take $j:V_\lambda\longrightarrow V_\theta$ witnessing $C^{(n)}$-extendibility of $\kappa$, with $\theta\in C^{(n)}$.  The parameters $a$ and $L$ are fixed.  Correctness of the source and target ranks, followed by elementarity, gives $j(M)\in\mathcal A$, and $j\restr|M|:M\to j(M)$ is $L$-elementary.  Finally,
\[
\rank(j(M))=j(\alpha)\geq j(\kappa)>\lambda>\alpha,
\]
so the two structures are distinct.
\end{proof}

We recall that a cardinal $\delta$ is rank-Berkeley if, for every $\zeta<\delta$ and every $\Theta>\delta$, there is a nontrivial elementary self-embedding $j:V_\Theta\to V_\Theta$ with $\zeta<\crit(j)<\delta$ and $j(\delta)=\delta$ \cite[Definition~4.4]{Mohammd}.

\begin{lemma}\label{lem:rank-Berkeley-descent}
$\ZF$ proves that the existence of a rank-Berkeley cardinal implies the existence of a set model of
\[
\ZF+\UE+\VP.
\]
\end{lemma}

\begin{proof}
 Let $\delta$ be rank-Berkeley.  Choose $\Theta>\delta+\omega$ and a nontrivial elementary embedding $j:V_\Theta\longrightarrow V_\Theta$ with critical point $\kappa_0<\delta$ and $j(\delta)=\delta$.  Put $\kappa_{m+1}=j(\kappa_m)$ and $\lambda=\sup_{m<\omega}\kappa_m$.  Inductively, $\kappa_m<\delta$ for all $m$, so $\lambda\leq\delta<\Theta$.  The sequence $s=\langle\kappa_m:m<\omega\rangle$ belongs to $V_\Theta$, and $j(s)$ is its tail.  Since $\lambda=\sup\operatorname{ran}(s)$ is a first-order relation between objects in the domain,
\[
j(\lambda)=\sup\operatorname{ran}(j(s))=\sup_{m<\omega}\kappa_{m+1}=\lambda.
\]
Hence $j\restr V_\lambda:V_\lambda\to V_\lambda$ is elementary.

If $f:a\to\kappa_0$ were cofinal for some $a\in V_{\kappa_0}$, then $j(f)$ would be cofinal in $\kappa_1$ although its range remained below $\kappa_0$.  Lemma~\ref{lem:ZF-rank-criterion} therefore gives $V_{\kappa_0}\models\ZF$.  Since $j\restr V_{\kappa_0}$ is the identity, Lemma~\ref{lem:restriction-elementarity} gives $V_{\kappa_0}\prec V_{\kappa_1}$.  Applying $j^m$ to this assertion gives
\[
V_{\kappa_m}\prec V_{\kappa_{m+1}}\qquad(m<\omega).
\]
The elementary-chain theorem yields $V_{\kappa_m}\prec V_\lambda$ for every $m$, and $V_\lambda\models\ZF$.

Fix $n\geq1$ and $i<\omega$.  We show inside $V_\lambda$ that $\kappa_i$ is $C^{(n)}$-extendible.  For $\mu<\lambda$, internal membership in $C^{(n)}$ is characterized by
\[
\mu\in(C^{(n)})^{V_\lambda}
\quad\text{if and only if}\quad
V_\mu\prec_{\Sigma_n}V_{\kappa_r}
\text{ for some }r\text{ with }\mu<\kappa_r,
\]
 because every $V_{\kappa_r}$ is elementary in $V_\lambda$.  Let $\mu>\kappa_i$ belong to $(C^{(n)})^{V_\lambda}$, and choose $r\geq i$ with $\mu+\omega<\kappa_r$ and $V_\mu\prec_{\Sigma_n}V_{\kappa_r}$.  Choose $s>r$ and $t>s$.  The set-sized restriction $e=j^t\restr V_{\kappa_s}:V_{\kappa_s}\longrightarrow V_{\kappa_{s+t}}$ is elementary by Lemma~\ref{lem:restriction-elementarity}, has rank below $\kappa_{s+t}+\omega$, and therefore belongs to $V_\lambda$.  Put $k=j^i(e)$, where this denotes the image of the set $e$ under $j^i$.  Then $k\in V_\lambda$, $k:V_{\kappa_{s+i}}\longrightarrow V_{\kappa_{s+t+i}}$, $\crit(k)=\kappa_i$, and $k(\kappa_i)=\kappa_{t+i}>\mu$.
Since $r\geq i$, we have $\kappa_r=j^i(\kappa_{r-i})$, and hence
\[
k(\kappa_r)=j^i\bigl(e(\kappa_{r-i})\bigr)=\kappa_{r+t}.
\]
Applying $k$ to $V_\mu\prec_{\Sigma_n}V_{\kappa_r}$ gives
\[
V_{k(\mu)}\prec_{\Sigma_n}V_{\kappa_{r+t}}.
\]
As $V_{\kappa_{r+t}}\prec V_\lambda$, the displayed characterization shows that $k(\mu)\in(C^{(n)})^{V_\lambda}$.  Because $k$ has critical point $\kappa_i$, $V_\lambda$ regards $\kappa_i$ as a cardinal.  Indeed, a surjection $f:\xi\to\kappa_i$ with $\xi<\kappa_i$ would be carried to a surjection from $\xi$ onto $k(\kappa_i)$, although $k(f)(\zeta)=f(\zeta)<\kappa_i$ for every $\zeta<\xi$.  The map $k\restr V_\mu\in V_\lambda$ is therefore the embedding required by $E_n(\kappa_i)$ for the chosen $\mu$.  Thus every $\kappa_i$ is $C^{(n)}$-extendible in $V_\lambda$.  Since the critical sequence is cofinal in $\lambda$ and $n$ was arbitrary, $V_\lambda\models\UE$.  Lemma~\ref{lem:OH-VP} now gives $V_\lambda\models\VP$.
\end{proof}

\begin{lemma}\label{lem:Cn-downward}
If $m>n\geq1$, then every $C^{(m)}$-extendible cardinal is $C^{(n)}$-extendible.
\end{lemma}

\begin{proof}
 It is enough to consider $m=n+1$.  Let $\kappa$ be $C^{(n+1)}$-extendible, and let $\lambda>\kappa$ belong to $C^{(n)}$.  Choose $\mu>\lambda$ in $C^{(n+1)}$ and an elementary embedding $j:V_\mu\longrightarrow V_\nu$ with $\crit(j)=\kappa$, $j(\kappa)>\mu$, and $\nu\in C^{(n+1)}$.  Because $\mu\in C^{(n+1)}$, we have $V_\mu\models\lambda\in C^{(n)}$.  Elementarity therefore gives
\[
V_\nu\models j(\lambda)\in C^{(n)}.
\]
 Since $\nu\in C^{(n+1)}$, $j(\lambda)\in C^{(n)}$ also holds in $V$.  For $n=1$, this follows from the $\Delta_2$ definition of $C^{(1)}$, while for $n\geq2$ it follows from the usual $\Pi_n$ definition of $C^{(n)}$.  Thus $j(\lambda)\in C^{(n)}$, and $j\restr V_\lambda:V_\lambda\longrightarrow V_{j(\lambda)}$ witnesses the required $C^{(n)}$-extendibility of $\kappa$.
\end{proof}

\begin{proposition}\label{prop:OH-pruning}
\[
\Con(\ZF+\VP)\quad\text{implies}\quad
\Con\bigl(\ZF+\VP+\UE+``\text{there are no rank-Berkeley cardinals}''\bigr).
\]
\end{proposition}

\begin{proof}
Put $T=\ZF+\VP$.  This is a recursively axiomatized theory: the instances of the parameterized Vop\v{e}nka scheme are generated effectively from formula codes.  The assertion $\mathsf{RB}$ that a rank-Berkeley cardinal exists is a single first-order sentence, since elementarity between set ranks is expressed by their canonical satisfaction relations.  The proof of Lemma~\ref{lem:rank-Berkeley-descent} can be formalized in $T$, uniformly in formula codes.  Hence $T$ proves $\mathsf{RB}\longrightarrow\Con(T)$.
If $T+\neg\mathsf{RB}$ were inconsistent, the deduction theorem would give $T\vdash\mathsf{RB}$ and hence $T\vdash\Con(T)$, contrary to G\"odel's second incompleteness theorem when $T$ is consistent.  Thus $T+\neg\mathsf{RB}$ is consistent.

{Now work in a model of $T+\neg\mathsf{RB}$.  Fix $n\geq2$.  {The scheme $\VP$ includes $\VP(\Pi_{n+1})$.}  Mohammd proves that if $\VP(\Pi_{n+1})$ holds while the $C^{(n)}$-extendible cardinals are bounded, then rank-Berkeley cardinals are unbounded \cite[Proposition~6.5 and Theorem~6.7]{Mohammd}.}  Thus the $C^{(n)}$-extendible cardinals are unbounded.  For $n=1$, boundedness of the $C^{(1)}$-extendibles implies boundedness of the $C^{(2)}$-extendibles by Lemma~\ref{lem:Cn-downward}; the $n=2$ case therefore gives unboundedly many $C^{(1)}$-extendibles as well.  Hence $\UE$ holds.
\end{proof}

\subsection*{{Local restoration of Choice}}

\begingroup

The next definitions record finite bounds on the L\'evy complexity needed for reflection.  Fix a G\"odel numbering of formulas, a primitive-recursive prenex-normal-form operation, and a primitive-recursive function $\ell(e)$ bounding the L\'evy complexity of the formula coded by $e$ and of its negation.  Let $\mathcal D$ be a finite list of formulas defining additional predicates, and write $e^{\mathcal D}$ for the result of eliminating those predicates.  Let $\Sigma$ be a finite set of codes for formulas whose free variables occur among the fixed ordered tuple
\[
(\eta,\delta,\theta,\rho,\alpha,x).
\]
Let $\Sigma^{\pm}$ be the closure of $\Sigma$ under subformulas and negation.  Its associated complexity bound is
\[
c_{\mathcal D}(\Sigma)=2+\max\Bigl(\{0\}\cup
\{\ell(e^{\mathcal D}):e\in\Sigma^{\pm}\}\Bigr).
\]
Thus two ranks in $C^{(c_{\mathcal D}(\Sigma))}$ agree on all formulas in $\Sigma$ after interpreting the predicates defined by $\mathcal D$.

Let $\chi_{\mathcal D,\Sigma}(\eta,\delta,\theta,\rho,\alpha,x)$ be the first-order formula
\[
\exists\bar\delta\,\exists\bar\theta\,\exists\bar\rho\,
\exists\bar\alpha\,\exists\bar x\,\exists J\;\Phi_{\mathcal D,\Sigma},
\]
where $\Phi_{\mathcal D,\Sigma}$ is the conjunction of
\[
\eta<\bar\delta<\bar\theta<\bar\theta+\omega<\bar\rho<
\delta<\theta<\theta+\omega<\rho,\qquad
\bar\alpha<\bar\theta,\qquad \bar x\in V_{\bar\rho},
\]
the assertion that $J:V_{\bar\rho}\to V_\rho$ is elementary with respect to the canonical satisfaction relations, the equations
\[
\crit(J)=\bar\delta,\quad J(\bar\delta)=\delta,\quad
J(\bar\theta)=\theta,\quad J(\bar\alpha)=\alpha,\quad J(\bar x)=x,
\]
the assertion $\bar\rho\in C^{(c_{\mathcal D}(\Sigma))}$, and the literal finite conjunction
\[
\bigwedge_{e\in\Sigma}
e^{\mathcal D}
[\eta,\bar\delta,\bar\theta,\bar\rho,\bar\alpha,\bar x],
\]
where the brackets denote primitive-recursive, capture-avoiding simultaneous substitution for the fixed tuple of free variables.  Predicate elimination, simultaneous substitution, finite conjunction, rank membership, and set-coded satisfaction are effective operations.  Hence the code of $\chi_{\mathcal D,\Sigma}$ is primitive recursive in $\mathcal D$ and $\Sigma$.  We use the following second complexity bound:
\[
s_{\mathcal D}(\Sigma)=2+\max\bigl\{c_{\mathcal D}(\Sigma),
\ell(\ulcorner\chi_{\mathcal D,\Sigma}\urcorner)\bigr\}.
\]
\endgroup

We write $V_\gamma\prec_{\Sigma_1^*}V$ for Spoerl's strengthened form of $\Sigma_1$-elementarity \cite[Definition~20]{Spoerl}.  We use Woodin's choiceless definition of supercompactness: a cardinal $\delta$ is supercompact if, whenever $\gamma>\delta$, $V_\gamma\prec_{\Sigma_1^*}V$, and $a\in V_\gamma$, there are $\bar\gamma<\delta$, $\bar\delta$, $\bar a\in V_{\bar\gamma}$, and an elementary embedding $j:V_{\bar\gamma+1}\to V_{\gamma+1}$ such that $\crit(j)=\bar\delta$, $j(\bar\delta)=\delta$, $j(\bar a)=a$, and $V_{\bar\gamma}\prec_{\Sigma_1^*}V$ \cite[Definition~21]{Spoerl}.

We use Woodin's local forcing for restoring Choice.  The following formulation combines his construction with the lifting lemma.  Spoerl gives a detailed exposition {\cite[Definition~43, \revision{Theorem~44}, and Lemma~45]{Spoerl}}, and Woodin's original construction appears in \cite[pp.~323--324]{Woodin}.  In Spoerl's presentation, the construction assigns a cardinal $\kappa_\beta$ to each stage $\beta$.  We call the cardinal through which the iteration has been defined its \emph{endpoint}.  A stage $\gamma$ is \emph{marked} if $\kappa_\gamma=\gamma$.

For an ordinal $\xi$, let $\DC_\xi$ denote dependent choice for sequences of length $\xi$.

\begin{theorem}[Woodin]\label{thm:Woodin-forcing}

Assume that $\AC$ fails, let $\mu$ be least such that $\DC_\mu$ fails, and let $\Lambda>\mu$ be supercompact in Woodin's choiceless sense.  There is a homogeneous forcing $\mathbb Q_\Lambda$, uniformly $\Sigma_3$-definable over $V_\Lambda$, such that whenever $G\subseteq\mathbb Q_\Lambda$ is generic:

\begin{enumerate}
\item $(V[G])_\Lambda\models\ZFC$;
\item {every marked $\gamma\leq\Lambda$ satisfies
\[
V_\gamma[G\cap\mathbb Q_\gamma]=(V[G])_\gamma;
\]
$\mathbb Q_\Lambda$ factors over $\mathbb Q_\gamma$ with the canonical projection, and every supercompact $\gamma\leq\Lambda$ is marked;}
\item suppose $J:V_{\bar\gamma+\omega}\to V_{\gamma+\omega}$ is elementary, $\crit(J)=\bar\delta$, $J(\bar\delta)=\delta>\bar\gamma$, $\gamma$ is marked, and $J(\mathbb Q_{\bar\gamma})=\mathbb Q_\gamma$.  Then $\bar\gamma$ is marked and $J$ lifts in $V[G]$ to $J^*:(V[G])_{\bar\gamma+1}\longrightarrow(V[G])_{\gamma+1}$.
\end{enumerate}
\end{theorem}

Let $\mathcal D_{\mathrm W}$ be a fixed finite list of formulas defining Woodin's iteration, its prefixes, orders, restriction maps, and marked stages.

For each $\gamma$, let $\Def(V_\gamma)$ be the collection of subsets of $V_\gamma$ definable over $(V_\gamma,\in)$ with parameters from $V_\gamma$.

\begin{lemma}\label{lem:supercompact-rank-ZF}
If $\gamma$ is supercompact in Woodin's choiceless sense, then $V_\gamma\models\ZF$.
\end{lemma}

\begin{proof}

By \cite[Definition~19 and the proof of Theorem~44]{Spoerl}, no map $V_\beta\to\gamma$ with $\beta<\gamma$ has unbounded range.  If $a\in V_\gamma$, choose $\beta<\gamma$ with $a\in V_\beta$; any cofinal map $a\to\gamma$ would extend by a constant value to $V_\beta$.  Lemma~\ref{lem:ZF-rank-criterion} now gives $V_\gamma\models\ZF$.
\end{proof}

\begin{lemma}\label{lem:Woodin-endpoint-forcing}
{For every $k\geq1$ there is an effectively obtained first-order formula $\vartheta_k(p,e,\vec\tau)$, independent of the endpoint, with the following property.  If Woodin's iteration is defined through a supercompact cardinal $\gamma$, then $\vartheta_k$ defines over $V_\gamma$ the forcing relation for $\mathbb Q_\gamma$ for formulas of L\'evy complexity at most $k$, and the truth lemma holds for every $\mathbb Q_\gamma$-generic filter over $(V_\gamma,\Def(V_\gamma))$.}
\end{lemma}

\begin{proof}
\begingroup

Fix the set-coded presentation used in \cite[Theorem~44(v)]{Spoerl}.  Given $\eta<\gamma$, $V_\xi$ computes the sequence $\langle\kappa_\beta:\beta\leq\eta\rangle$ correctly for every sufficiently large inaccessible $\xi<\gamma$.  An induction on the stages of the recursion in \cite[Definition~43]{Spoerl} shows that, for all sufficiently large inaccessible $\xi<\gamma$, the construction in $V_\xi$ agrees with that in $V_\gamma$ through stage $\eta$.  Thus the initial segment through $\eta$ belongs to some $V_\xi$ below $\gamma$.  Since this segment is the same in every sufficiently large $V_\xi$, the sequence of initial segments is definable over $V_\gamma$.

Since the endpoint is marked and inaccessible, it cannot be an inverse-limit stage, as that clause would give $\gamma=\kappa_\gamma=(\sup_{\alpha<\gamma}\kappa_\alpha)^+$.  Hence it is a direct-limit stage and $\sup_{\alpha<\gamma}\kappa_\alpha=\gamma$.  Let $e_{\kappa_\alpha,\gamma}$ be the canonical iteration embedding and set
\[
P_\alpha=e_{\kappa_\alpha,\gamma}``\mathbb Q_{\kappa_\alpha}.
\]
Replacement in $V_\gamma$ makes each $P_\alpha$ a set.  Under the factorization of \cite[Theorem~44(iii)(a)]{Spoerl}, $e_{\kappa_\alpha,\gamma}$ is the first-coordinate embedding, and reductions are obtained by strengthening that coordinate while leaving the tail unchanged.  Thus the $P_\alpha$ form a definable increasing sequence of complete subforcings with
$\mathbb Q_\gamma=\bigcup_{\alpha<\gamma}P_\alpha$.  Lemma~\ref{lem:supercompact-rank-ZF} gives $V_\gamma\models\ZF$.  By \cite[Lemma~6.6]{HolyClassForcing}, this union embeds densely into a $V_\gamma$-complete Boolean algebra $\mathbb B\in\Def(V_\gamma)$ approachable by projections.

Theorems~6.4 and~4.3 of \cite{HolyClassForcing} yield the atomic forcing relation and the full forcing theorem, respectively.  Pass to an outer model that contains the given generic and in which $V_\gamma$ is countable.  The structure $(V_\gamma,\Def(V_\gamma))$, the filtration, and its completeness properties are unchanged, so the forcing relation and truth lemma obtained there apply to the given generic.

For fixed standard $k$, recursion from the atomic relation through the codes of formulas of L\'evy complexity at most $k$ gives a single first-order definition, uniform in $\gamma$.  Assigning a fixed value to all remaining codes defines $\vartheta_k$.
\endgroup
\end{proof}

{For $k\geq1$, fix $\vartheta_k$ as in Lemma~\ref{lem:Woodin-endpoint-forcing} and put $r_k=c_{\mathcal D_{\mathrm W}}(\{\vartheta_k\})$.}

\begin{lemma}\label{lem:finite-windows}
Suppose $\theta<\Lambda$ are supercompact in Woodin's choiceless sense, Woodin's iteration is defined through $\Lambda$, and $\theta,\Lambda\in C^{(r_k)}$.  If $G\subseteq\mathbb Q_\Lambda$ is generic, then
\[
(V[G])_\theta\prec_{\Sigma_k}(V[G])_\Lambda.
\]
\end{lemma}

\begin{proof}
  Both endpoints are marked by Theorem~\ref{thm:Woodin-forcing}.  Let $G_\theta$ be induced by the complete embedding of $\mathbb Q_\theta$ into $\mathbb Q_\Lambda$; in the nested presentation, $G_\theta=G\cap\mathbb Q_\theta$.  Every dense subclass of $\mathbb Q_\theta$ definable over $V_\theta$ with set parameters is externally a set, since $V_\theta$ and its satisfaction relation are sets.  Completeness therefore gives that $G_\theta$ is generic over $(V_\theta,\Def(V_\theta))$.  Since $\theta$ is marked,
  \[
  V_\theta[G_\theta]=(V[G])_\theta.
  \]
  Thus every $x\in(V[G])_\theta$ is $\tau^{G_\theta}$ for some $\mathbb Q_\theta$-name $\tau\in V_\theta$.

Since $\theta,\Lambda\in C^{(r_k)}$, predicate elimination for $\mathcal D_{\mathrm W}$ gives, for every formula $\varphi$ of complexity at most $k$, every tuple of $\mathbb Q_\theta$-names $\vec\tau\in V_\theta$, and every $q\in\mathbb Q_\theta$,
\begin{equation}\label{eq:endpoint-force-agreement}
q\Vdash_{\mathbb Q_\theta}\varphi(\vec\tau)
\quad\text{if and only if}\quad
q\Vdash_{\mathbb Q_\Lambda}\varphi(\vec\tau).
\end{equation}
The same equivalence holds for $\neg\varphi$ by the closure under negation in the definition of $c_{\mathcal D_{\mathrm W}}$.

Suppose first that $(V[G])_\Lambda\models\varphi(\vec\tau^{G_\theta})$.  By Lemma~\ref{lem:Woodin-endpoint-forcing}, choose $p\in G$ with $p\Vdash_{\mathbb Q_\Lambda}\varphi(\vec\tau)$.  The dense set used below is
\[
D=\{q\leq p\restr\theta:q\in\mathbb Q_\theta\text{ and }q\Vdash_{\mathbb Q_\theta}\varphi(\vec\tau)\}.
\]
We show that $D$ is dense below $p\restr\theta$.  Otherwise, some $q_0\leq p\restr\theta$ has no extension in $D$.  The negation clause then gives $q_0\Vdash_{\mathbb Q_\theta}\neg\varphi(\vec\tau)$, and \eqref{eq:endpoint-force-agreement} gives $q_0\Vdash_{\mathbb Q_\Lambda}\neg\varphi(\vec\tau)$.  The reduction property gives a common extension of $p$ and $q_0$, contradicting the two forcing statements.  The class $D\cup\{r\in\mathbb Q_\theta:r\perp p\restr\theta\}$ is definable and dense in $\mathbb Q_\theta$.  Genericity and $p\restr\theta\in G_\theta$ therefore give $D\cap G_\theta\ne\varnothing$.  For $q\in D\cap G_\theta$, forcing soundness from Lemma~\ref{lem:Woodin-endpoint-forcing} gives $(V[G])_\theta\models\varphi(\vec\tau^{G_\theta})$.

Conversely, if $(V[G])_\theta\models\varphi(\vec\tau^{G_\theta})$, Lemma~\ref{lem:Woodin-endpoint-forcing} gives $q\in G_\theta$ such that $q\Vdash_{\mathbb Q_\theta}\varphi(\vec\tau)$.  Equation~\eqref{eq:endpoint-force-agreement} transfers this assertion to $\mathbb Q_\Lambda$, and forcing soundness gives $(V[G])_\Lambda\models\varphi(\vec\tau^{G_\theta})$.  This proves the required $\Sigma_k$-elementarity.
\end{proof}

\begin{lemma}\label{lem:finite-reflection}
{Let $S\geq s_{\mathcal D}(\Sigma)$, suppose $E_S(\delta)$, and let $\eta<\delta<\theta<\theta+\omega<\rho$, with $\rho\in C^{(c_{\mathcal D}(\Sigma))}$.  If $\alpha<\theta$, $x\in V_\rho$, and
\[
V\models e^{\mathcal D}
[\eta,\delta,\theta,\rho,\alpha,x]
\quad\text{for every }e\in\Sigma,
\]
then $\chi_{\mathcal D,\Sigma}(\eta,\delta,\theta,\rho,\alpha,x)$ holds.}
\end{lemma}

\begin{proof}
Choose $\mu\in C^{(S)}$ above $\rho+\omega$ and let $k:V_\mu\to V_\nu$ witness $E_S(\delta)$.  Thus $\nu\in C^{(S)}$, $\crit(k)=\delta$, and $k(\delta)>\mu>\rho$.  The graph of $k\restr V_\rho$ has rank below $k(\rho)+\omega=k(\rho+\omega)<\nu$, and Lemma~\ref{lem:restriction-elementarity} shows that this restriction is elementary from $V_\rho$ to $V_{k(\rho)}$.  Use the following tuple:
\[
(\bar\delta,\bar\theta,\bar\rho,\bar\alpha,\bar x,J)
=(\delta,\theta,\rho,\alpha,x,k\restr V_\rho).
\]
Thus $V_\nu\models{\chi_{\mathcal D,\Sigma}}(\eta,k(\delta),k(\theta),k(\rho),k(\alpha),k(x))$.
Here $\rho<k(\delta)$ and $k(\eta)=\eta$.  By assumption, the original tuple satisfies the conjunction over $\Sigma$, and the $C^{(S)}$-correctness of $V_\nu$ makes the same conjunction true there.  Elementarity of $k$ pulls the assertion back to $V_\mu$.  Since $\mu\in C^{(S)}$ and the formula and its negation have L\'evy complexity below $S$, the reflected assertion is true in $V$.
\end{proof}

\begin{lemma}\label{lem:ordinary-correct-supercompact}
Every $E_s$-cardinal is $\Sigma_{s+2}$-correct.  Moreover, there is a fixed $d_{\mathrm W}<\omega$ such that every $E_s$-cardinal is supercompact in Woodin's sense whenever $s\geq d_{\mathrm W}$.
\end{lemma}

\begin{proof}
An $E_s$-cardinal is $s$-choiceless extendible in the sense of Mohammd: the ordinary witness has critical point $\delta>\xi$ for every prescribed $\xi<\delta$.  Its $\Sigma_{s+2}$-correctness is therefore \cite[Proposition~2.2]{Mohammd}.

Fix formulas $\sigma_0(\xi)$ and $\sigma_1(\delta,\gamma,a)$ expressing $V_\xi\prec_{\Sigma_1^*}V$ and the small-embedding clause in Spoerl's Definition~21, and choose $d_{\mathrm W}$ above the L\'evy complexities of these formulas and their negations.  Let $s\geq d_{\mathrm W}$ and assume $E_s(\delta)$.  Given $\gamma>\delta$ with $V_\gamma\prec_{\Sigma_1^*}V$ and $a\in V_\gamma$, choose an $E_s$-witness $k:V_\mu\to V_\nu$ with $\mu\in C^{(s)}$ above $\gamma+\omega$.  By Lemma~\ref{lem:restriction-elementarity}, $k\restr V_{\gamma+1}$ witnesses in $V_\nu$ the required assertion for $k(\gamma),k(\delta),k(a)$.  Elementarity reflects this to $V_\mu$, and the choice of $d_{\mathrm W}$ makes the reflected map and the $\Sigma_1^*$-correctness of its source absolute to $V$.  Thus $\delta$ is supercompact in Woodin's sense.
\end{proof}

We now apply Woodin's forcing at an $E_s$-cardinal for large enough $s$ and use Lemmas~\ref{lem:finite-reflection} and~\ref{lem:ordinary-correct-supercompact} to obtain a $\ZFC$ rank model satisfying any given finite subset of $\ZFC+\VP$.

\begin{theorem}\label{thm:finite-restoration}

For every $N\geq1$ there is an effectively determined integer $t_N$ such that the following holds.  Suppose $V\models\ZF+\UE+\neg\AC$, and let $\mu$ be least such that $\DC_\mu$ fails.  Let $\Lambda>\mu$ be an $E_{t_N}$-cardinal, and let $G\subseteq\mathbb Q_\Lambda$ be generic for Woodin's forcing.  Then $M=(V[G])_\Lambda$ satisfies $\ZFC$ and, for every $1\leq n\leq N$, has a proper class of $C^{(n)}$-extendible cardinals.
\end{theorem}

\begin{proof}
For $\gamma\leq\Lambda$, write $M_\gamma=(V[G])_\gamma$ for the rank-$\gamma$ part of $M$.  Let $k=N+1$ and put $r=r_k$.  Let $\Xi_N$ be the finite set of formulas asserting that a stage $\theta$ belongs to $C^{(r)}$, is supercompact in Woodin's sense, and has the prefix, order, and restriction maps defined by $\mathcal D_{\mathrm W}$.  The two complexity bounds needed below are
\[
c_N=c_{\mathcal D_{\mathrm W}}(\Xi_N),\qquad
s_N=\max\{r,d_{\mathrm W},s_{\mathcal D_{\mathrm W}}(\Xi_N)\}.
\]
Let $t_N$ be four more than the maximum of $s_N$ and the L\'evy complexities of $E_{s_N}(\xi)$, the assertion that the $E_{s_N}$-cardinals are unbounded, membership in $C^{(c_N)}$, and the assertion that $C^{(c_N)}$ is unbounded.  Once the coding is fixed, $t_N$ is primitive recursive in $N$.

By Lemma~\ref{lem:ordinary-correct-supercompact}, $\Lambda\in C^{(t_N+2)}$ and $\Lambda$ is supercompact in Woodin's sense.  Theorem~\ref{thm:Woodin-forcing} gives $M\models\ZFC$.  Since $\UE$ asserts that the $E_{s_N}$-cardinals are unbounded, the choice of $t_N$ implies that they are unbounded below $\Lambda$; the same argument applies to $C^{(c_N)}\cap\Lambda$.

For each $1\leq n\leq N$, the Bagaria--Poveda characterization says that, in $M$, a cardinal $\delta$ is $C^{(n)}$-extendible if and only if, for a proper class of $\theta\in(C^{(n+1)})^M$ and every $\alpha<\theta$, there are $\bar\delta<\bar\theta<\delta$, $\bar\alpha<\bar\theta$, and an elementary embedding $e:M_{\bar\theta}\to M_\theta$ such that $\bar\theta\in(C^{(n+1)})^M$, $\crit(e)=\bar\delta$, $e(\bar\delta)=\delta$, and $e(\bar\alpha)=\alpha$ \cite[Corollary~2.7]{BagariaPoveda}.

Fix $\eta<\Lambda$ above $\mu$, and choose an $E_{s_N}$-cardinal $\delta>\eta$.  Let $\theta>\delta$ be another such cardinal and let $\alpha<\theta$.  Choose $\rho\in C^{(c_N)}\cap\Lambda$ above $\theta+\omega$, and let $x\in V_\rho$ code $\mathbb Q_\theta$ with its order and restriction maps.  Lemma~\ref{lem:ordinary-correct-supercompact} and Woodin's recursion show that these objects satisfy $\Xi_N$.  Lemma~\ref{lem:finite-reflection} gives
\[
\eta<\bar\delta<\bar\theta<\bar\theta+\omega<\bar\rho<\delta
\]
and $J_0:V_{\bar\rho}\to V_\rho$ such that
\[
\crit(J_0)=\bar\delta,\quad J_0(\bar\delta)=\delta,\quad
J_0(\bar\theta)=\theta,\quad J_0(\bar\alpha)=\alpha,
\quad J_0(\mathbb Q_{\bar\theta})=\mathbb Q_\theta.
\]
The reflected instances of $\Xi_N$ give $\bar\theta\in C^{(r)}$ and show that $\bar\theta$ is supercompact in Woodin's sense.

Set $J=J_0\restr V_{\bar\theta+\omega}$.  Lemma~\ref{lem:restriction-elementarity} shows that $J:V_{\bar\theta+\omega}\longrightarrow V_{\theta+\omega}$ is elementary.

The map $J$ satisfies the hypotheses of Theorem~\ref{thm:Woodin-forcing}(3), with $J(\bar\delta)=\delta>\bar\theta$ and $J(\mathbb Q_{\bar\theta})=\mathbb Q_\theta$.  Hence it lifts to $J^*:M_{\bar\theta+1}\longrightarrow M_{\theta+1}$.
Lemma~\ref{lem:restriction-elementarity} shows that $e=J^*\restr M_{\bar\theta}$ is elementary from $M_{\bar\theta}$ to $M_\theta$.  It has critical point $\bar\delta$, maps $\bar\delta$ to $\delta$ and $\bar\alpha$ to $\alpha$, and belongs to $M$ because its graph has rank below $\theta+\omega<\Lambda$.  The ranks $\bar\theta$, $\theta$, and $\Lambda$ are marked and belong to $C^{(r)}$.

Lemma~\ref{lem:finite-windows}, applied twice, gives
\[
M_{\bar\theta}\prec_{\Sigma_{N+1}}M
\quad\text{and}\quad
M_\theta\prec_{\Sigma_{N+1}}M.
\]

Consequently $\bar\theta,\theta\in(C^{(n+1)})^M$ for every $n\leq N$.  As in the proof of Lemma~\ref{lem:rank-Berkeley-descent}, $M_{\bar\theta}$ regards the critical point $\bar\delta$ as a cardinal, so $M_\theta$ regards $\delta=e(\bar\delta)$ as a cardinal.  Any contrary witness in $M$ already belongs to $M_\theta$, so $\delta$ is a cardinal in $M$.

For the fixed $\delta$, the ordinal $\theta$ can be chosen arbitrarily high below $\Lambda$, and $\alpha<\theta$ was arbitrary, so the characterization shows that $\delta$ is $C^{(n)}$-extendible for every $n\leq N$.  The cardinal $\delta$ was arbitrary above $\eta$, and $\eta$ was arbitrary above $\mu$; hence these cardinals are unbounded below $\Lambda=\Ord^M$, so each class is proper in $M$.
\end{proof}

\begingroup

\begin{lemma}\label{lem:finite-language-reduction}
For every finite set $F_0$ of instances of the arbitrary-language Vop\v{e}nka scheme, there is an effectively determined $N\geq1$ such that
\[
\ZFC+\VP(\Pi_{N+1})\vdash F_0.
\]
\end{lemma}

\begin{proof}
Working in $\ZFC$, well-order each of the finitely many set-sized languages and apply the standard coding of structures in an arbitrary set-sized language by structures in one fixed finite relational language; the coding preserves and reflects elementary embeddings \cite[pp.~2--3]{BrookeTaylor}.  The well-order of the language and all parameters in the original instance may be retained as parameters in the definition of the coded class.  For each fixed instance, the definitions of the coding and decoding have some finite L\'evy complexity.  Taking the maximum of these finitely many bounds gives $N$.
\end{proof}
\endgroup

\begin{theorem}[\revision{Woodin}]\label{thm:consistency-transfer}
\[
\Con(\ZF+\VP)\quad\text{if and only if}\quad\Con(\ZFC+\VP).
\]
\end{theorem}

\begin{proof}
By Proposition~\ref{prop:OH-pruning}, take a countable model
\[
\mathfrak M\models\ZF+\VP+\UE+``\text{there are no rank-Berkeley cardinals}''.
\]
Let $F$ be a finite subset of the first-order theory $\ZFC+\VP$.  If $\mathfrak M\models\AC$, then $\mathfrak M\models F$.

Now assume that $\mathfrak M\models\neg\AC$.  Choose $N\geq1$ for the finitely many instances of $\VP$ occurring in $F$, as in Lemma~\ref{lem:finite-language-reduction}.  Choose $t_N$ as in Theorem~\ref{thm:finite-restoration} and choose in $\mathfrak M$ a sufficiently large $E_{t_N}$-cardinal $\Lambda$.  Let $H\subseteq(\mathbb Q_\Lambda)^{\mathfrak M}$ be $\mathfrak M$-generic and put $W=(\mathfrak M[H])_\Lambda$.  By Theorem~\ref{thm:finite-restoration}, $W\models\ZFC$ and $W$ has a proper class of $C^{(N)}$-extendible cardinals.  {Bagaria's characterization \cite[Theorem~4.3(2), Corollary~4.7, Theorem~4.11, and the paragraph following Theorem~4.12]{BagariaCn} therefore gives the $\VP(\Pi_{N+1})$ scheme in $W$ for classes of structures in a fixed finite relational language.  By Lemma~\ref{lem:finite-language-reduction}, $W$ satisfies every instance of $\VP$ occurring in $F$; since $W\models\ZFC$, it follows that $W\models F$.}

Every finite subset of $\ZFC+\VP$ is therefore satisfiable.  By first-order compactness, $\ZFC+\VP$ has a model.  The converse implication follows by forgetting Choice.
\end{proof}

This proves Theorem~B.

\begin{remark}\label{rem:consistency-not-uniform}

For a finite set of axioms $F$, the rank extension used in the proof may depend on $F$.  Compactness therefore proves equiconsistency, but does not give a class-generic $\ZFC+\VP$ extension of the original model.  Question~\ref{q:strong-restoration} asks whether every countable model admits one.
\end{remark}
\endgroup

\section{{Independence of Choice principles}}

Theorem~\ref{thm:preservation} applies to the classical symmetric constructions in this section and the next two; Theorem~\ref{thm:consistency-transfer} gives a ground model of $\ZFC+\VP$ when one is required.

\begin{corollary}\label{cor:choice}
\[
\Con(\ZF+\VP)\quad\text{implies}\quad
\Con(\ZF+\VP+\neg\AC+\neg\DC).
\]
Consequently $\AC$ is independent of $\ZF+\VP$, relative to $\Con(\ZF+\VP)$.
\end{corollary}

\begin{proof}
Let $\mathfrak M$ be a countable model of $\ZF+\VP$, and form its basic Cohen symmetric extension.  The canonical set $A=\{x_n:n<\omega\}$ of Cohen reals is infinite and Dedekind-finite \cite[Section~5.3]{Jech}.  Hence $\AC$ and $\DC$ fail.  Theorem~\ref{thm:preservation} preserves $\VP$, while Theorem~\ref{thm:consistency-transfer} gives the opposite consistency implication for $\AC$.
\end{proof}

\begingroup

\begin{lemma}\label{lem:omega1-Cohen}
Assume $\ZFC+\VP$.  There is a set-sized symmetric extension satisfying
\[
\ZF+\VP+\DC+\neg\AC.
\]
\end{lemma}

\begin{proof}
Let $\mathbb P_1$ be the forcing of countable partial functions from $\omega_1\times\omega_1$ to $2$, ordered by reverse inclusion.  Let the full permutation group of $\omega_1$ act on the first coordinate, and let the normal filter be generated by pointwise stabilizers of countable subsets of $\omega_1$.  The forcing is $\sigma$-closed and the filter is $\sigma$-complete, so the symmetric extension satisfies $\DC$ by Karagila's preservation theorem \cite[Lemma~3.1]{KaragilaDC}.  Theorem~\ref{thm:preservation} gives $\ZF+\VP$.

For $\xi<\omega_1$, let $a_\xi$ be the $\xi$th generic subset of $\omega_1$ and put $A=\{a_\xi:\xi<\omega_1\}$.  The name for $A$ is invariant, while the name for $a_\xi$ has support $\{\xi\}$, so $A$ belongs to the symmetric extension.  Suppose that $p$ forces a hereditarily symmetric name $\dot f$, supported by a countable $E\subseteq\omega_1$, to enumerate $A$.  Enlarge $E$ to contain the first-coordinate support of $p$.  Choose $\xi\notin E$, $q\leq p$, and an ordinal $\gamma$ such that
\[
q\forcesS\dot f(\check\gamma)=\dot a_\xi.
\]
Choose $\zeta$ outside $E$ and outside the countable set of first coordinates occurring in $q$, and let $\pi$ transpose $\xi$ and $\zeta$.  Then $\pi$ fixes $\dot f$ and $p$, and
\[
\pi q\forcesS\dot f(\check\gamma)=\dot a_\zeta.
\]
Since $q$ and $\pi q$ are compatible, choose a common extension.  It forces $\dot a_\xi\neq\dot a_\zeta$ while forcing both names to equal $\dot f(\check\gamma)$, a contradiction.  Thus $A$ is not well-orderable, and $\AC$ fails.
\end{proof}

\begin{theorem}\label{thm:DC-notAC}
\[
\Con(\ZF+\VP)\quad\text{implies}\quad
\Con(\ZF+\VP+\DC+\neg\AC).
\]
\end{theorem}

\begin{proof}

Let $\mathfrak M\models\ZF+\VP$ be countable.  If $\mathfrak M\models\AC$, apply Lemma~\ref{lem:omega1-Cohen}.  If $\mathfrak M\models\neg\AC+\neg\DC$, choose a singular weakly LS cardinal $\kappa$ by Corollary~\ref{cor:singular-LS}.  Usuba's forcing $\Col(V_\kappa)$ forces $\DC$ \cite[Proposition~3.6]{Usuba}, while Theorem~\ref{thm:preservation} preserves $\VP$; if the extension satisfies $\AC$, apply Lemma~\ref{lem:omega1-Cohen} once more.  The remaining case already satisfies $\DC+\neg\AC$.
\end{proof}

\begin{corollary}\label{cor:DC-independent}
Relative to $\Con(\ZF+\VP)$, the axiom $\DC$ is independent of
\[
\ZF+\VP+\neg\AC.
\]
\end{corollary}

\begin{proof}
Corollary~\ref{cor:choice} gives a model with $\neg\DC$, and Theorem~\ref{thm:DC-notAC} gives one with $\DC$, from the same consistency hypothesis.
\end{proof}

Let $\RR$ denote the rigid relation principle: every set carries a binary relation with no nonidentity automorphism.

\begin{corollary}\label{cor:RR}
The following consistency implications hold:
\[
\begin{aligned}
\Con(\ZF+\VP)&\quad\text{implies}\quad\Con(\ZF+\VP+\neg\RR),\\
\Con(\ZF+\VP)&\quad\text{implies}\quad
\Con(\ZF+\VP+\RR+\neg\AC+\neg\DC).
\end{aligned}
\]
\end{corollary}

\begin{proof}
Hamkins and Palumbo prove that every model of $\ZF$ has a set-sized symmetric extension in which $\RR$ fails, and that $\RR$ holds in the basic Cohen model \cite[Theorems~5 and~6]{HamkinsPalumbo}.  Apply Theorem~\ref{thm:preservation} to the first construction over a model of $\ZF+\VP$.  For the second implication, first use Theorem~\ref{thm:consistency-transfer} to obtain a model of $\ZFC+\VP$, and then use the Cohen model from the proof of Corollary~\ref{cor:choice}; it satisfies $\neg\AC+\neg\DC$, and Theorem~\ref{thm:preservation} preserves $\VP$.
\end{proof}
\endgroup

\begingroup

\section{The Feferman--L\'evy model}\label{sec:FL}

\begin{theorem}\label{thm:FL}
If $\ZF+\VP$ is consistent, then so is $\ZF+\VP+\neg\AC+\neg\DC$ together with the assertion that $\mathbb R$ is a countable union of countable sets.
\end{theorem}

\begin{proof}
By Theorem~\ref{thm:consistency-transfer}, it suffices to start with $V\models\ZFC+\VP$.  {Let $\mathbb P_0$ be trivial, let $N_0$ consist of the canonical names for the ground-model reals, and choose an infinite regular cardinal $\kappa_1>|N_0|$.  Recursively, after $\kappa_1,\ldots,\kappa_n$ have been chosen, define the finite-stage forcing by
\[
\mathbb P_n=\prod_{1\leq m\leq n}^{\mathrm{fin}}\Coll(\omega,\kappa_m),
\]
let $N_n$ be the set of nice $\mathbb P_n$-names for reals, and choose an infinite regular cardinal $\kappa_{n+1}>\max\{\kappa_n,|N_n|\}$.}  Following the standard Feferman--L\'evy construction \cite[Theorem~10.6]{Jech}, use the full forcing
\[
{\mathbb P=\prod_{1\leq m<\omega}^{\mathrm{fin}}\Coll(\omega,\kappa_m).}
\]
Let $\Gamma$ consist of the permutations acting independently on the $\omega$-coordinate of each column, and let $\mathcal F$ be generated by the subgroups $H_n$ fixing the first $n$ columns pointwise, with $H_0=\Gamma$.  Let $G$ be generic and $W=\HS^G$.  Proposition~\ref{prop:symmetric-ZF} and Theorem~\ref{thm:preservation} give
\[
W\models\ZF+\VP.
\]
For $n<\omega$, let $G_n$ be the restriction of $G$ to columns $1,\ldots,n$, with $G_0$ trivial.  We use the notation
\[
{R_n=\mathbb R^{V[G_n]}.}
\]

If $x\in W$ is a real, $\dot x$ has support $H_n$, and $p\forcesS\check k\in\dot x$, the usual support argument gives
\[
p\restr n\forcesS\check k\in\dot x.
\]
If $q\leq p\restr n$ forced the opposite decision, choose $\pi\in H_n$ moving the finitely many coordinates of $p$ above column $n$ away from those of $q$.  Then $\pi p$ and $q$ are compatible and $\pi$ fixes $\dot x$, a contradiction.  {Therefore
\[
x=\{k\in\omega:\exists q\in G_n\ (q\Vdash_{\mathbb P}\check k\in\dot x)\},
\]
and ordinary and symmetric forcing agree in this atomic formula for these hereditarily symmetric names.  Hence $x\in V[G_n]$.  Conversely, every $\mathbb P_n$-name is hereditarily fixed by $H_n$, and hence $V[G_n]\subseteq W$.}  Consequently,
\[
\mathbb R^W=\bigcup_{n<\omega}R_n.
\]

{Regard the set $N_n$ fixed above as a set of $\mathbb P$-names.  The names used below are}
\[
\dot R_n=\{\langle\dot y,1_{\mathbb P}\rangle:\dot y\in N_n\},
\qquad
\dot S=\{\langle\langle\check n,\dot R_n\rangle^\bullet,1_{\mathbb P}\rangle:n<\omega\}.
\]
Every automorphism preserves the first $n$ columns as a set and permutes $N_n$, so $\dot R_n$ and $\dot S$ are invariant.  Their constituent names are supported by $H_n$, hence $\dot S\in\HS$ and
\[
\dot S^G=\langle R_n:n<\omega\rangle.
\]
Hence the entire sequence lies in $W$.

Finally, each $R_n$ is countable in $W$.  {Fix $n<\omega$ and, in $V$, well-order $N_n$.  Mapping each real in $R_n$ to the least index of a nice name evaluating to it gives, in $V[G_n]\subseteq W$, an injection $e_n:R_n\longrightarrow |N_n|^V$, where $|N_n|^V<\kappa_{n+1}$.
{The next generic column adds a surjection $g_{n+1}:\omega\to\kappa_{n+1}$.  Extending $e_n^{-1}$ by a fixed value off $\operatorname{ran}(e_n)$ and composing with $g_{n+1}$ gives a surjection from $\omega$ onto $R_n$.  For $n=0$, the filter $G_0$ is trivial, $N_0$ is the chosen set of canonical names, and $g_1$ is the first generic column.}}

Thus $W$ contains the sequence $\langle R_n:n<\omega\rangle$ of countable sets with union $\mathbb R^W$.  By Cantor's theorem, $\mathbb R^W$ is uncountable.  Hence countable choice fails, and so does $\DC$.
\end{proof}
\endgroup

\section{The full Solovay model}

Suppose for this section that $V\models\ZFC+\VP$.  By the final clause of Lemma~\ref{lem:ZF-ranks}, fix a measurable cardinal $\kappa$ and a nonprincipal $\kappa$-complete ultrafilter $U$ on $\kappa$.  Let
$G\subseteq{\Coll}(\omega,{<}\kappa)$ be $V$-generic, put
$R=(\omega^\omega)^{V[G]}$, and define
\begin{equation}\label{eq:full-Solovay}
\mathcal S\!ol=\HOD_{V\cup R}^{V[G]},
\end{equation}
where individual elements of $V$ and individual members of $R$ are allowed as parameters.  We use $\LM$, $\BP$, and $\PSP$ for the assertions that every set of reals is, respectively, Lebesgue measurable, has the Baire property, and has the perfect set property.

\begin{lemma}\label{lem:Solovay-regularity}
The model $\mathcal S\!ol$ satisfies
$\ZF+\DC+\LM+\BP+\PSP$, and $\kappa=\omega_1^{\mathcal S\!ol}$.
\end{lemma}

\begin{proof}
The model in \eqref{eq:full-Solovay} is the full Solovay model in the precise convention of \cite[Definition~4 and the following paragraph]{SakaiTanno}.  Since a measurable cardinal is inaccessible in $\ZFC$, Solovay's theorem gives
$\mathcal S\!ol\models\ZF+\DC+\LM+\BP+\PSP$ \cite[Theorem~1]{Solovay}.

The collapse makes every $\alpha<\kappa$ countable in $\mathcal S\!ol$.  Since $\kappa$ is inaccessible, ${\Coll}(\omega,{<}\kappa)$ is $\kappa$-c.c. and preserves the uncountability of $\kappa$.  Hence $\kappa=\omega_1^{\mathcal S\!ol}$.
\end{proof}

\begin{lemma}\label{lem:Solovay-localization}
If $x\in\mathcal S\!ol$ is a set of ordinals, then
$x\in V[G_\beta]$ for some $\beta<\kappa$, where
$G_\beta=G\cap{\Coll}(\omega,{<}\beta)$.  Moreover, if
$\langle A_\xi:\xi<\delta\rangle\in\mathcal S\!ol$, where
$\delta<\kappa$ and every $A_\xi\subseteq\kappa$, then the whole sequence belongs to one such bounded-stage extension.
\end{lemma}

\begin{proof}
Let $x\in\mathcal S\!ol$ be a set of ordinals, and choose an ordinal $\theta$ with $x\subseteq\theta$.  There are a formula $\varphi$, ordinal parameters $\vec\gamma$, a parameter $a\in V$, and a real $r\in V[G]$ such that
\[
\alpha\in x\quad\text{if and only if}\quad
V[G]\models\varphi(\alpha,\vec\gamma,a,r)
\qquad(\alpha<\theta).
\]
A nice name for $r$ has support of size less than $\kappa$, and hence bounded support, so $r\in V[G_\beta]$ for some $\beta<\kappa$.  Let $\mathbb P^\beta$ be the quotient forcing over $V[G_\beta]$.  Since $\mathbb P^\beta$ is weakly homogeneous and its automorphisms fix the parameters,
\[
x=\{\alpha<\theta:
1_{\mathbb P^\beta}\Vdash
\varphi(\check\alpha,\check{\vec\gamma},\check a,\check r)\}.
\]
The definability of the forcing relation and Separation in $V[G_\beta]$ therefore give $x\in V[G_\beta]$.  This is the localization argument for the full Solovay model \cite[Definition~4, Fact~5, and the paragraph following Fact~5]{SakaiTanno}.

For the final assertion, code the sequence by
\[
C=\{\kappa\cdot\xi+\alpha:\xi<\delta,\ \alpha\in A_\xi\}.
\]
Then $C\in\mathcal S\!ol$, so $C\in V[G_\beta]$ for some $\beta<\kappa$.  Since
\[
A_\xi=\{\alpha<\kappa:\kappa\cdot\xi+\alpha\in C\},
\]
the sequence $\langle A_\xi:\xi<\delta\rangle$ belongs to $V[G_\beta]$.
\end{proof}

\begin{lemma}\label{lem:Solovay-measure}
In $\mathcal S\!ol$, the set
\[
U^*=\{A\in\mathcal P(\kappa)^{\mathcal S\!ol}:
(\exists Y\in U)\,Y\subseteq A\}
\]
is a nonprincipal, internally $\kappa$-complete ultrafilter on $\kappa$.
\end{lemma}

\begin{proof}
Since $U\in V\subseteq\mathcal S\!ol$, Separation gives $U^*\in\mathcal S\!ol$.  For $\beta<\kappa$, define
\[
U_\beta=\{X\in\mathcal P(\kappa)^{V[G_\beta]}:
(\exists Y\in U)\,Y\subseteq X\}.
\]
The forcing ${\Coll}(\omega,{<}\beta)$ has ground-model cardinality less than $\kappa$, so \cite[Theorem~2.5(2)]{Banerjee} shows that $U_\beta$ is a $\kappa$-complete ultrafilter.  If $A\subseteq\kappa$ belongs to $\mathcal S\!ol$, choose $\beta<\kappa$ with $A\in V[G_\beta]$ by Lemma~\ref{lem:Solovay-localization}.  Then $A\in U^*$ if and only if $A\in U_\beta$, so exactly one of $A$ and $\kappa\setminus A$ belongs to $U^*$.

To verify internal $\kappa$-completeness, let $\delta<\kappa$ and
$\langle A_\xi:\xi<\delta\rangle\in\mathcal S\!ol$, with every
$A_\xi\in U^*$.  By Lemma~\ref{lem:Solovay-localization}, the sequence
$\langle A_\xi:\xi<\delta\rangle$ belongs to $V[G_\beta]$ for some
$\beta<\kappa$.  Since each $A_\xi\in U^*$, we have $A_\xi\in U_\beta$.  The $\kappa$-completeness of $U_\beta$ therefore gives
\[
\bigcap_{\xi<\delta}A_\xi\in U_\beta.
\]
The intersection belongs to $\mathcal S\!ol$, and its membership in $U_\beta$ therefore implies that it belongs to $U^*$.  Finally, no singleton belongs to $U^*$ because $U$ is nonprincipal.  Thus $U^*$ is nonprincipal and internally $\kappa$-complete.
\end{proof}

\begin{lemma}\label{lem:Solovay-symmetric}
The model $\mathcal S\!ol$ is the symmetric extension arising from a set-sized symmetric system in $V$.
\end{lemma}

\begin{proof}
Let $\mathbb B$ be the Boolean completion of ${\Coll}(\omega,{<}\kappa)$, let $H\subseteq\mathbb B$ be the induced generic, and define the Boolean class name below using the condition-first convention of \cite{KaragilaSchilhan}:
\[
\dot R_{\mathrm{can}}=
\{(b,\tau):\tau\text{ is a }\mathbb B\text{-name and }
0<b\leq\|\tau\in\check\omega^{\check\omega}\|_{\mathbb B}\}.
\]
Then $\dot R_{\mathrm{can}}^H=R$.  Equivariance of Boolean values gives
$\pi\dot R_{\mathrm{can}}=\dot R_{\mathrm{can}}$ for every
$\pi\in\Aut(\mathbb B)$.  Moreover, $\dot R_{\mathrm{can}}$ self-reflects in the sense of \cite[Definitions~9.1--9.2]{KaragilaSchilhan}: for any fixed ground-model real $r_0$, both $\dot R_{\mathrm{can}}$ and $\check r_0$ have stabilizer $\Aut(\mathbb B)$, and
$1_{\mathbb B}\Vdash\check r_0\in\dot R_{\mathrm{can}}$.

Consequently, by \cite[Proposition~9.4]{KaragilaSchilhan}, the symmetric extension arising from
$\mathcal O(\mathbb B,\dot R_{\mathrm{can}})$ is exactly
\[
\HOD_{V\cup R}^{V[H]}=\mathcal S\!ol,
\]
using the displayed self-reflecting name.  This symmetric system is set-sized in $V$: $\mathbb B$ and $\Aut(\mathbb B)$ are sets, and its generated normal filter is a definable subset of
$\mathcal P(\Aut(\mathbb B))$.
\end{proof}

\begin{corollary}\label{cor:Solovay}
$\Con(\ZF+\VP)$ implies the consistency of
\[
\ZF+\VP+\DC+\LM+\BP+\PSP+\neg\AC+
``\omega_1\text{ carries a nonprincipal }\omega_1\text{-complete ultrafilter}.''
\]
\end{corollary}

\begin{proof}
By Theorem~\ref{thm:consistency-transfer}, take the starting model to satisfy $\ZFC+\VP$.  By Lemma~\ref{lem:Solovay-symmetric} and Theorem~\ref{thm:preservation},
$\mathcal S\!ol\models\VP$.  Lemma~\ref{lem:Solovay-regularity} gives $\ZF+\DC+\LM+\BP+\PSP$ and identifies $\kappa$ with $\omega_1^{\mathcal S\!ol}$, while Lemma~\ref{lem:Solovay-measure} gives the nonprincipal, internally $\kappa$-complete ultrafilter.  The standard Solovay argument gives $\neg\AC$.
\end{proof}

\begingroup

\section{{Limits of Choice restoration}}\label{sec:class-forcing}

Question~\ref{q:strong-restoration} asks whether every countable model of $\ZF+\VP$ has a class-generic extension satisfying $\ZFC+\VP$.  Under the stated forcing-theorem hypotheses, Proposition~\ref{prop:class-forcing-criteria} characterizes when a class forcing preserves $\VP(\Pi_m)$.  Proposition~\ref{prop:SVC-restoration} gives a positive answer under $\mathsf{SVC}$.  Corollary~\ref{cor:wellorderable-no-choice-restoration} rules out well-orderable set forcing, while Proposition~\ref{prop:countability-essential} gives an $\aleph_1$-sized counterexample to the version without countability.

\begingroup

\begin{proposition}\label{prop:class-forcing-criteria}
Let $(M,\mathcal X)\models\GBC$ be a countable transitive two-sorted model, {where $\GBC$ denotes G\"odel--Bernays set theory with Global Choice}, and let $\mathbb P\in\mathcal X$ be a class forcing with a largest condition.  Assume that the forcing theorem holds for first-order formulas, that the relevant forcing relations belong to $\mathcal X$, that a generic may be chosen through every condition, and that every generic extension has the same ordinals and satisfies $\ZFC$.  For standard $m\geq1$, let $K_m(\kappa)$ mean that $\kappa$ is supercompact if $m=1$, and $C^{(m-1)}$-extendible if $m\geq2$.  Then
\[
1_{\mathbb P}\Vdash\VP(\Pi_m)\quad\text{if and only if}\quad
\forall p\in\mathbb P\ \forall\eta\in\Ord^M\ 
\exists q\leq p\ \exists\kappa>\eta\ 
q\Vdash K_m(\check\kappa).
\]
Here the left side means that every coded instance of $\VP(\Pi_m)$ is forced.  Consequently, $\mathbb P$ forces the full first-order scheme exactly when the right side holds for every standard $m\geq1$.
\end{proposition}

\begin{proof}
Over $\ZFC$, $\VP(\Pi_m)$ is equivalent to the unboundedness of the cardinals satisfying $K_m$ \cite[Theorem~4.3(2), Corollary~4.7, Theorem~4.11, and the paragraph following Theorem~4.12]{BagariaCn}; arbitrary set-sized languages reduce to the usual coding under Choice \cite[Lemma~23]{BrookeTaylor}.  Suppose first that $\VP(\Pi_m)$ is forced, and fix $p$ and $\eta$.  Choose a generic $G$ containing $p$.  Some $\kappa>\eta$ in $M[G]$ satisfies $K_m$, and the preservation of ordinals gives $\kappa\in M$.  By the assumed truth lemma and directedness of $G$, some $q\leq p$ forces $K_m(\check\kappa)$.

Conversely, for each $\eta$, the conditions forcing a $K_m$-cardinal above $\eta$ are dense, so every generic extension satisfies $\VP(\Pi_m)$.  For each instance, the hypotheses that a generic can be chosen through every condition and that the truth lemma holds imply that the conditions forcing it are dense.  Hence $1_{\mathbb P}$ forces every instance of $\VP(\Pi_m)$.
\end{proof}
\endgroup

\begin{question}\label{q:strong-restoration}
{For every countable first-order model $M\models\ZF+\VP$, do there exist a class part $\mathcal X$ over $M$, a class forcing $\mathbb P\in\mathcal X$, and an externally $(M,\mathcal X)$-generic filter $G\subseteq\mathbb P$ such that the quotient-of-names extension satisfies $M[G]\models\ZFC+\VP$?}
\end{question}

\begin{question}\label{q:selective-destruction}
{For each standard $n\geq1$, is it relatively consistent with $\ZFC+\VP(\Pi_{n+1})$ that a class forcing, definable in the language of set theory from set parameters, preserves $\ZFC$ and $\VP(\Pi_n)$ but destroys $\VP(\Pi_{n+1})$?}
\end{question}

Two related results bear on Question~\ref{q:selective-destruction}.  {For $n\geq2$, suitable $\Delta_2$-definable class iterations preserve $\VP(\Pi_n)$ over models of $\VP(\Pi_{n+1})$ \cite[Theorem~9.6]{BagariaPoveda}.}  {On the other hand, Hamkins shows that the first-order scheme need not imply Vop\v{e}nka's principle for arbitrary classes \cite[Theorem~11]{Hamkins}.}

Blass introduced the axiom of small violations of Choice, $\mathsf{SVC}$.  We use his characterization that, over $\ZF$, $\mathsf{SVC}$ is equivalent to the existence of a set forcing which forces $\AC$ \cite{Blass}; see also \cite[Section~3]{KaragilaDC}.

\begin{proposition}\label{prop:SVC-restoration}
Let $M\models\ZF+\VP+\mathsf{SVC}$.  There is a set forcing $\mathbb P\in M$ such that every $M$-generic extension by $\mathbb P$ satisfies $\ZFC+\VP$.  In particular, if $M$ is countable, Question~\ref{q:strong-restoration} has a positive answer for $M$.
\end{proposition}

\begin{proof}
By Blass's characterization, choose $\mathbb P\in M$ with $1_{\mathbb P}\Vdash\AC$.  Set forcing preserves $\ZF$, and the trivial-symmetry case of Theorem~\ref{thm:preservation} preserves $\VP$.  Thus $1_{\mathbb P}\Vdash\ZFC+\VP$.
\end{proof}

\begin{lemma}\label{lem:wellorderable-forcing}
Let $\mathbb P$ be a well-orderable set forcing, $X$ a ground-model set, and suppose that for some ordinal $\kappa$ and name $\dot f$,
\[
1_{\mathbb P}\Vdash\dot f:\check\kappa\twoheadrightarrow\check X.
\]
Then $X$ is well-orderable in the ground model.
\end{lemma}

\begin{proof}
Fix a well-order of $\mathbb P\times\kappa$.  For $x\in X$, let $d(x)$ be the least pair $(p,\alpha)$ such that $p\Vdash\dot f(\check\alpha)=\check x$.  Since $1_{\mathbb P}$ forces $\dot f$ to be surjective, Replacement yields a function $d:X\longrightarrow\mathbb P\times\kappa$.
If $d(x)=d(y)=(p,\alpha)$, then $p$ forces both $\dot f(\check\alpha)=\check x$ and $\dot f(\check\alpha)=\check y$, so $x=y$.  Thus $d$ is injective, and hence $X$ is well-orderable.
\end{proof}

\begin{corollary}\label{cor:wellorderable-no-choice-restoration}
A well-orderable set forcing cannot restore Choice over a model of $\neg\AC$.
\end{corollary}

\begin{proof}
Let a generic extension by a well-orderable $\mathbb P$ satisfy $\AC$, and fix a ground-model set $X$.  In the extension, choose a surjection $f:\kappa\twoheadrightarrow X$ from an ordinal $\kappa$.  By the truth lemma, there are a condition $p\in\mathbb P$ and a name $\dot f$ such that $p\Vdash\dot f:\check\kappa\twoheadrightarrow\check X$.  After restricting $\dot f$ below $p$, Lemma~\ref{lem:wellorderable-forcing} applied to the well-orderable forcing $\mathbb P\mathbin{\restriction}p$ shows that $X$ is well-orderable in the ground model.  Thus every ground-model set is well-orderable, and the ground model satisfies $\AC$.
\end{proof}

\begin{proposition}\label{prop:countability-essential}
Assume $\Con(\ZF+\VP)$.  There is a model $N$ of cardinality $\aleph_1$ satisfying $\ZF+\VP+\neg\AC$ which has no proper generic extension satisfying $\ZF$.  Consequently, the analogue of Question~\ref{q:strong-restoration} for all set-sized models is false.
\end{proposition}

\begin{proof}
By Corollary~\ref{cor:choice}, take a countable $M_0\models\ZF+\VP+\neg\AC$.  Enayat's theorem gives an elementary extension $M_0\prec N$ of cardinality $\aleph_1$ having no proper end extension to a model of $\ZF$ \cite[Theorem~5.18]{Enayat}.  Elementarity gives $N\models\ZF+\VP+\neg\AC$.  {The atomic forcing recursion shows that old sets acquire no new members in a quotient-of-names extension.  Hence any such extension that adds a set is both proper and an end extension.}  Therefore no such extension of $N$ can satisfy $\ZF$.
\end{proof}
\endgroup

\section*{\revision{Concluding remarks}}

\revision{We propose collaborations between mathematicians and
contributors exploring set theory with LLM assistance to develop
detailed arguments and understand their relation to existing work.
Early-career researchers can be valuable partners, contributing to
the proofs and drawing on their connections to specialists to
clarify the methods and history.}

\revision{Carefully checked drafts can support ongoing discussion of the
arguments and their provenance. These exchanges can help identify
and document results previously known only informally. We welcome
questions, corrections, and pointers to earlier work concerning this
paper. This feedback can inform subsequent revisions, with appropriate
credit for both the original results and contributions to their
written proofs.}

\end{document}